\documentclass[12pt]{amsart}
\usepackage{tikz-cd}
\usepackage{amsmath}
\usepackage{fourier}
\usepackage{amssymb}
\usepackage{amscd}
\usepackage{amsthm}
\usepackage[centertags]{amsmath}
\usepackage{amsfonts}
\usepackage{newlfont}
\usepackage{graphicx}
\usepackage{amsfonts, amssymb}
\usepackage{mathrsfs}
\usepackage{latexsym}
\usepackage{tikz}
\usepackage{verbatim}
\usepackage[all]{xy}
\usepackage{enumitem}
\usepackage[colorlinks=true,linkcolor=colorref,citecolor=colorcita,urlcolor=colorweb]{hyperref}
\definecolor{colorcita}{RGB}{21,86,130}
\definecolor{colorref}{RGB}{5,10,177}
\definecolor{colorweb}{RGB}{177,6,38}

\numberwithin{subsection}{section}

\newtheorem{theorem}{Theorem}[section]

\newtheorem{proposition}[theorem]{Proposition}
\newtheorem{corollary}[theorem]{Corollary}
\newtheorem{lemma}[theorem]{Lemma}

\theoremstyle{definition}

\theoremstyle{remark}

\newcommand{\dis}{\displaystyle}

\newcommand{\Bb}{\mathcal B}

\newcommand{\NN}{\mathbb N}
\newcommand{\EE}{\mathbb E}

\newcommand{\RR}{\mathbb R}

\usepackage{mathtools}

\DeclareMathOperator{\id}{\mathrm{id}}
\DeclarePairedDelimiter\floor{\lfloor}{\rfloor}

\newcommand{\bi}{\mathbf i}
\newcommand{\bj}{\mathbf j}

\newcommand{\sinc}{\operatorname{sinc}}

\DeclareMathOperator{\sid}{Sid}

\def\N{\mathbb{ N}}
\def\R{\mathbb{ R}}

\begin{document}
\title[Asymptotic insights for spaces of functions on the Boolean cube]{Asymptotic insights for Projection, Gordon--Lewis and Sidon Constants\\ in Boolean Cube Function Spaces}

\author[Defant]{A.~Defant}
\address{%
Institut f\"{u}r Mathematik,
Carl von Ossietzky Universit\"at,
26111 Oldenburg,
Germany}
\email{defant$@$mathematik.uni-oldenburg.de}

\author[Galicer]{D.~Galicer}
\address{Departamento de Matem\'{a}tica,
Facultad de Cs. Exactas y Naturales, Universidad de Buenos Aires and IMAS-CONICET. Ciudad Universitaria, Pabell\'on I
(C1428EGA) C.A.B.A., Argentina}
\email{dgalicer$@$dm.uba.ar}

\author[Mansilla]{M.~Mansilla}
\address{Departamento de Matem\'{a}tica,
Facultad de Cs. Exactas y Naturales, Universidad de Buenos Aires and IAM-CONICET. Saavedra 15 (C1083ACA) C.A.B.A., Argentina}
\email{mmansilla$@$dm.uba.ar}

\author[Masty{\l}o]{M.~Masty{\l}o}
\address{Faculty of Mathematics and Computer Science, Adam Mickiewicz University, Pozna{\'n}, Uniwersytetu
\linebreak
Pozna{\'n}skiego 4,
61-614 Pozna{\'n}, Poland}
\email{mieczyslaw.mastylo$@$amu.edu.pl}

\author[Muro]{S.~Muro}
\address{FCEIA, Universidad Nacional de Rosario and CIFASIS, CONICET, Ocampo $\&$ Esmeralda, S2000 Rosario, Argentina}
\email{muro$@$cifasis-conicet.gov.ar}

\date{}

\thanks{The research of the fourth author was supported by the National Science Centre (NCN), Poland,
Project 2019/33/B/ST1/00165; the second, third and fifth author were supported by CONICET-PIP 11220200102336 and PICT 2018-4250. The research of the fifth author is additionally supported by ING-586-UNR}

\begin{abstract}
 The main aim of this work is to study important local Banach space constants for Boolean cube function spaces.
Specifically, we focus on $\mathcal{B}_{\mathcal{S}}^N$,  the finite-dimensional Banach space of all real-valued functions defined on the $N$-dimensional Boolean cube $\{-1, +1\}^N$ that have Fourier--Walsh
expansions supported on a fixed~family $\mathcal{S}$ of subsets of $\{1, \ldots, N\}$.   Our investigation centers on the  projection, Sidon and Gordon--Lewis constants of this function space.  We combine tools from different areas to derive exact formulas and asymptotic estimates of these parameters for special types of families  $\mathcal{S}$ depending on the dimension $N$ of the Boolean cube and other complexity  characteristics of the support set $\mathcal{S}$.  Using local  Banach space theory, we establish the intimate relationship among these three important constants.

\end{abstract}

\subjclass[2020]{Primary:  06E30; 46B06; 46B07.  Secondary: 42A16; 43A75}

\keywords{Projection constants, Boolean functions, local theory of Banach spaces.}
\maketitle


\section{Introduction}

We study important Banach space invariants such as projection, Sidon and Gordon--Lewis constants of certain natural subspaces of the
Banach space of all real-valued functions defined on the  $N$-dimensional Boolean cube $\{-1, +1\}^N$.
 Recall that every function  $f\colon \{-1, +1\}^N \to \mathbb{R}$ has a~Fourier--Walsh
expansion of the form
\[
f(x) = \sum_{S\subset \{1,\ldots,N\}}\widehat{f}(S) x^{S}\,,
\]
where for each $(x_1, \ldots, x_N)\in \{-1, +1\}^N$, $x^S:=\prod_{k \in S} x_k$ is a Walsh function. The set of all $S$ for which 
$\widehat{f}(S)\neq 0$
is called the spectrum of $f$.

More precisely, given  a~set $\mathcal{S}$ of subsets in  $\{1,\ldots,N\}$,  we consider the space
$\mathcal{B}^N_{\mathcal{S}}$ of all functions $f\colon \{-1, +1\}^{N} \to \mathbb{R}$ with
Fourier--Walsh expansions supported on $\mathcal{S}$, that is, $\widehat{f}(S) \neq 0$
only if $S \in \mathcal{S}$. Endowed with the supremum norm on $\{-1,+1\}^{N}$, this is
a~finite dimensional  Banach space.
Our main goal then is to find asymptotically correct estimates for the projection constant
$\boldsymbol{\lambda}\big(\mathcal{B}_{\mathcal{S}}^N\big)$, and to link this invariant with
other important invariants from Fourier analysis and local Banach space theory like the Sidon,
 unconditional basis  or Gordon--Lewis constants.

We note that the Fourier analysis of
functions on Boolean cubes is essential in  theoretical computer science, and plays a key role
in combinatorics, random graph theory, statistical physics, Gaussian geometry, the theories of metric
spaces/Banach spaces, learning theory,  or  social choice theory (see, e.g., \cite{o2008some} and
\cite{o2014analysis} and the references therein). Moreover, the last decades show  growing interest
for applications of Boolean functions in the context of quantum algorithms complexity and  quantum
information \cite{beals2001quantum}. For recent important developments and applications
in this direction see also \cite{aaronson2014need,bourgain2002distribution,defant2019fourier,eskenazis2022learning,
Eskenazis2023Low-degree,rouze2022quantum,slote2023noncommutative,volberg2022noncommutative}
.

We mainly focus on  the Banach space $\mathcal{B}_{\leq  d}^N :=\mathcal{B}^N_{\{S \colon |S|\leq d\}}$,
that is  all real-valued functions $f$ on the compact abelian group $\{-1, +1\}^N$ with Fourier transforms
$\hat{f}$ supported on all subsets of $[N]$ with cardinality $\leq d$, and similarly on the Banach space
$\mathcal{B}_{=d}^N :=\mathcal{B}_{\{S \colon |S|= d\}}$.

The study of complemented subspaces $X$ of a~Banach space $Y$ and their projection constants
has a~long history going  back to the beginning of operator theory in Banach spaces.
For a~general overview of the state of art of the  theory of projection
constants in Banach spaces, we refer to the excellent  monograph  \cite{tomczak1989banach}
by Tomaczak-Jaegermann and references therein.


We recall that if $X$  is a~subspace of a~Banach space $Y$, then the relative projection
constant of $X$ in $Y$ is defined by
\[
\boldsymbol{\lambda}(X, Y) = \inf\big\{\|P\|: \,\, P\in \mathcal{L}(Y, X),\,\, P|_{X} = \id_X\big\}\,,
\]
where $\id_X$ is the identity operator on $X$, and the (absolute) projection constant of  $X$ is
given by
\begin{equation}\label{definition}
\boldsymbol{\lambda}(X) := \sup \,\,\boldsymbol{\lambda}(I(X),Z)\,,
\end{equation}
where the supremum is taken over all Banach spaces $Z$ and isometric embeddings $I\colon X \to Z$.
The following straightforward result shows the intimate link  between projection constants and
extensions of linear operators: For every Banach space $Y$ and its subspace $X$ one has
\begin{align*}
\boldsymbol{\lambda}(X, Y) = \inf\big\{c>0: \,\,\text{$\forall\, T \in \mathcal{L}(X, Z)$ \,\, $\exists$\,\, an extension\,
$\widetilde{T}\in \mathcal{L}(Y, Z)$\, with $\|\widetilde{T}\| \leq c\,\|T\|$}\big\}\,.
\end{align*}

Drawing from Rudin's averaging technique 
from \cite{rudin1962projections}
for estimating the projection constant (dating back to the 1960s), an adapted   technique 
for spaces of trigonometric polynomials on compact abelian groups was devised in \cite{defant2024projection}. 
We apply this new perspective into the framework of functions on Boolean cubes. As a consequence, we in
Theorem~\ref{bool-int} see that
\begin{equation*}
\boldsymbol{\lambda}\big(\mathcal{B}_\mathcal{S} ^{N}\big) =\frac{1}{2^N}
\sum_{  x \in\{-1, +1\}^N }\big| \sum_{S \in \mathcal{S}} x^s \big|.
\end{equation*}
In principle, this integral can be calculated  with a~computer - at least for concrete well-structured
families $\mathcal{S}$ in $[N]:=\{1,2,\dots,N\}$. Nevertheless, if $N$ is large or  the set $\mathcal{S}$
is 'too  big', this might get unfeasible. Therefore, it is important to study the asymptotic order of
$\boldsymbol{\lambda}\big(\mathcal{B}_\mathcal{S} ^{N}\big)$ in the dimension  $N$ and/or other parameters
quantifying the complexity of $\mathcal{S}$.

Among others, we show in Theorem~\ref{thm: limit bool homog} that
\[
\lim_{N\to \infty}N^{-d/2} \boldsymbol{\lambda}\big(\mathcal{B}_{\mathcal{S}}^N\big)
= \frac{1}{\sqrt{2\pi}} \int_{\mathbb{R}} \frac{|h_d(t)|}{d!} \exp(-t^2/2)\,d\!t\,,
\]
where $h_d$ is the $d$-th Hermite polynomial.

While this outcome can be derived from a very interesting identity by Beckner \cite{beckner1975inequalities} from the mid-seventies—utilizing it as a black box—alongside the central limit theorem, we derived it (without prior knowledge of it) from a specific technique to study the combinatorial structure of the index set $\mathcal{S}$ involved. This technique is grounded on ideas introduced in \cite{galicer2021monomial} and \cite{mansilla2019thesis}, where it was employed to examine sets of monomial convergence of spaces of holomorphic functions in high dimensions.
We highlight that this combinatorial tool holds significant potential, extending its applicability to diverse contexts. In this manuscript, we present both methods for obtaining the result.


The Sidon constant plays a pivotal role in Fourier analysis, providing crucial insights into the behavior and convergence properties of Fourier series and related transformations. 
Recall that, given  $\mathcal{S} \subset [N]$,  the Sidon constant of the characters $(\chi_S)_{S \in \mathcal{S} }$ in the group $\{ -1, 1\}^{N}$ is the best constant $C >0$ such that for all $f \in \mathcal{B}^N_{\mathcal{S} }$,
\begin{equation}\label{sidon}
\sum_{S \in \mathcal{S} } |\hat{f}(S)| \,\le \,C \| f \|_\infty\,,
\end{equation}
and we in the following are going to denote this constant by   $\boldsymbol{\sid}(\mathcal{B}^N_{\mathcal{S} })$. 
In the context of spaces of Boolean functions, this invariant establishes a connection between the Fourier coefficients and the supremum norm of the function. Essentially, it quantifies how the distribution of the Fourier coefficients across different frequencies influences the overall behavior of the function, as reflected by its supremum norm. This relationship provides valuable insights into the harmonic structure and complexity of Boolean functions, aiding in the analysis and understanding of their properties and computational characteristics.

We demonstrate how to estimate the Sidon constant of spaces of Boolean functions for a specific set $\mathcal S$ and the projection constant of another space associated with the 'reduced form' of $\mathcal S$. To achieve this, we employ various techniques from local analysis of Banach spaces, including the connection with the so-called Gordon--Lewis constant, tensorial techniques, and certain symmetrization/desymmetrization ideas.

 The main result here is Theorem~\ref{BGL3} (and more generally Theorem~\ref{BGL2}), where we show that in the  homogeneous
case the Sidon constant of  $\mathcal{B}_{=d}^N $ may be estimated by $\boldsymbol{\lambda}\big(\mathcal{B}^N_{=d-1}\big)$
up to a constant $C^d$, where $C >0$ is universal. 
The proof goes a detour relating first the Sidon constant with the 
Gordon--Lewis constant (Theorem~\ref{BGL1}) and then in a~ second  step the Gordon--Lewis constant 
with the projection constant (Theorem~\ref{gl_versus_proj}).

Summarizing, we use local Banach space theory to show that the Sidon,  Gordon--Lewis and projection constant of every Banach space
$\mathcal{B}_{\mathcal{S}}^N$ are intimately  linked -- a connection of which all three of these fundamental constants  benefit
abundantly.

Using a recent variant of the Bohnenblust-Hille inequality for functions on Boolean cubes from \cite{defant2019fourier}, we conclude the manuscript by relating the Sidon constant of the space $\mathcal{B}^N_{\mathcal{S}}$ with the 'size' of the support set $\mathcal{S}$. 

\section{Preliminaries}
Standard notation from Banach space theory as e.g., used in the monographs \cite{diestel1995absolutely,LT1,
pisier1986factorization, tomczak1989banach, wojtaszczyk1996banach} is going to be needed. In this note, we basically only consider  real
Banach spaces.

\subsection{Functions on  Boolean cubes}
Throughout the paper we for $N \in \mathbb{N}$ call $\{-1, +1\}^N$ the $N$-dimen\-sional Boolean cube. A~Boolean function is any function
$f\colon \{-1, +1\}^N \rightarrow \{-1, +1\}$, more generally,   functions  $f : \{-1, +1\}^N  \rightarrow \mathbb{R}$ are said to be  `real valued functions on the  $N$-dimensional Boolean cube'.

The study of  real valued functions on Boolean cubes is deeply influenced by Fourier analysis. Considering the $N$-dimensional Boolean cube
$\{-1, +1\}^N$ as a~compact abelian group endowed with the coordinatewise product and the  discrete topology (so the Haar measure is given
by the normalized counting measure), we may apply the machinery given by abstract harmonic analysis.

In particular, the integral or expectation of  each function $f:\{-1, +1\}^N \rightarrow \mathbb{R}$ is given by
\[
\mathbb{E}\big[ f \big] := \frac{1}{2^{N}} \sum_{x \in \{-1, +1\}^N}{f(x)}\,.
\]
The characters on $\{-1, +1\}^N$  are the so-called Walsh functions defined as
\[
\chi_{S}\colon \{-1, +1\}^N \rightarrow \{ -1,1\} \, ,
\hspace{3mm} \chi_{S}(x) = x^{S}:= \prod_{k \in S}{x_{k}} \,\, \,\,\,\, \text{for} \,\,\,\,\,\,x \in \{-1, +1\}^N,
\]
where $S \subset [N]:=\{ 1, \ldots, N\}$ and  $\chi_{\emptyset}(x):= 1$ for each $x\in \{-1, +1\}^N$.
This allows to associate to each function $f:\{-1, +1\}^N \rightarrow \R$ its Fourier--Walsh expansion
\begin{equation*}\label{equa:FourierExpansionShort}
f(x) = \sum_{S \subset [N]}{\widehat{f}(S) \, x^S}\, , \hspace{3mm} x \in \{-1, +1\}^N\,,
\end{equation*}
where  $\widehat{f}(S) = \mathbb{E}\big[ f \cdot \chi_{S} \big]$ are the Fourier coefficients.

Given $d \in \mathbb{N}$, we say that  $f$ has degree $d $ whenever  $\widehat{f}(S) = 0$ for all $|S| > d$, and $f$ is  $d$-homogeneous whenever additionally  $\widehat{f}(S) = 0$ provided  $|S| \neq d$; where, as usual,  $\vert S \vert$ stands for the cardinality of the set $S$.

For  every set  $\mathcal{S}$ of subsets $S \subset [N] $, we define the linear space $\mathcal{B}^{N}_{\mathcal{S}}$ of all  functions $f\colon \{-1, +1\}^N \rightarrow \mathbb{R}$,
which have Fourier--Walsh expansions supported on $\mathcal{S}$. Endowed  with the supremum norm
$\|\cdot\|_\infty$ on the $N$-dimensional Boolean cube, this space  turns into a  Banach space.
We mainly concentrate  on the following special classes of functions on the $N$-dimensional
Boolean cube:
\begin{itemize}
\setlength\itemsep{0.8em}
\item $\mathcal{B}^{N}$ :=  all functions $f:\{-1, +1\}^N \rightarrow \mathbb{R}$,
\item $\mathcal{B}^{N}_{=d}$  :=  all  $d$-homogeneous $f:\{-1, +1\}^N \rightarrow \mathbb{R}$,
\item $\mathcal{B}^{N}_{\leq d}$ :=  all $f:\{-1, +1\}^N \rightarrow \mathbb{R}$ with  degree less or equal $d$\,.
\end{itemize}
Obviously, we have the isometric identity
\begin{align} \label{identitiesI}
\mathcal{B}^N = \ell_\infty\big(\{-1, +1\}^N\big)\,, \,\,\,\,\,\,\,\,f \mapsto (f(x))_{x \in\{-1, +1\}^N}\,.
\end{align}

\subsection{Functions on Boolean cubes  vs tetrahedral polynomials}

As usual we call the elements $\alpha = (\alpha_i)$ in $\mathbb{N}_0^{(\mathbb{N})}$ (all finite sequences in $\mathbb{N}_0$) multi indices,
and $|\alpha | = \sum \alpha_i$ is their so-called order. For $N \in \mathbb{N}$ and $d \in \mathbb{N}_0$
\begin{align*}
\text{$\Lambda(d,N) := \{ \alpha \in \mathbb{N}_0^N \colon |\alpha | = d \}$, \quad \,$\Lambda(\leq d,N)
:= \{ \alpha \in \mathbb{N}_0^N \colon |\alpha | \leq  d \}$ }
\end{align*}
denote  the sets of multi indices which are  $d$-homogeneous and of order $\leq d$, respectively.
A multi index $\alpha = (\alpha_i)\in \mathbb{N}_0^{(\mathbb{N})}$ is said to be   tetrahedral whenever
each entry $\alpha_i$ is either $0$ or $1$. We denote by $\Lambda_T$ the subset of all tetrahedral multi indices in $\mathbb{N}_0^{(\mathbb{N})}$. For  $d \leq N$ we let
\begin{align*}
\text{$\Lambda_T(d,N) := \Lambda(d,N) \cap \Lambda_T$ , \quad \, $\Lambda_T(\leq d,N) := \Lambda(\leq d,N) \cap \Lambda_T$. }
\end{align*}
It  turns out to be convenient to have an equivalent description of $\Lambda(d,N)$. We write
\begin{align*}
& \mathcal{M}(d,N) =[N]^d\,,
\\
&
\mathcal{J}(d,N) = \big\{\bj = (j_1, \ldots, j_d) \in \mathcal{M}(d,N)\colon \, j_1 \leq \ldots \leq j_d\big\}\,.
\end{align*}
Then there obviously is a~canonical bijection between $\mathcal{J}(d,N)$ and $\Lambda(d,N)$. Indeed,
assign to $\bj \in \mathcal{J}(d,N) $ the multi index $\alpha \in \Lambda(d,N)$ given by $\alpha_r = |\{k \colon \bj_k = r\}|,\, 1 \leq r \leq N$, and conversely assign to  each $\alpha \in \Lambda(d,N)$
the index  $\bj \in \mathcal{J}(d,N)$, where $j_1 =\ldots =j_{\alpha_1} =1$, $j_{\alpha_1 +1} =\ldots =j_{\alpha_1+\alpha_2} =2,\, \dots$

On $\mathcal{M}(d,N)$ we consider the equivalence relation: $\bi  \sim \bj$ if there is a permutation
$\sigma$ on $\{1, \ldots,d\}$ such that $(i_1, \ldots, i_d) = (j_{\sigma(1)}, \ldots, j_{\sigma(d)})$.
The equivalence class of $\bi \in \mathcal{M}(d,N)$ is denoted by $[\bi]$, and its cardinality  by
$|[\bi]|$.  We write $|[\alpha]| :=  |[\bj]|$ provided that $\bj$ is associated with $\alpha$, and
in this case have that
\begin{equation} \label{scholz}
    |[\alpha]| =  |[\bj]| =\frac{d!}{\alpha!}\,,
\end{equation}
where $\alpha!:= \alpha_1! \cdot \ldots \cdot \alpha_N!\,.$

We consider finite polynomials $P: \mathbb{R}^N \to \mathbb{R}$, i.e., polynomials of the form
$P(x) = \sum_{\alpha \in F} c_\alpha x^\alpha, \, x \in \mathbb{R}$, where $F$ is a finite set of multi indices in $\mathbb{N}_0^{(\mathbb{N})}$.
We write $\mathcal{P}_d(\ell_\infty^N)$  for all $d$-homogenous  polynomials, so polynomials  of the form
$P(x) = \sum_{\alpha \in \Lambda(d,N) } c_\alpha x^\alpha, \, x \in \mathbb{R}$. The space $\mathcal{P}_d(\ell_\infty^N)$, together with the supremum norm
taken on the unit ball of $\ell_\infty^N$,  forms a Banach space.
A polynomial  is said to be tetrahedral, whenever $c_\alpha =0 $  for all $\alpha \notin \Lambda_T$.
The subspace of all tetrahedral $d$-homogeneous polynomials is denoted
by $\mathcal{T}_d(\ell_\infty^N)$, that is, all polynomials  of the form
$P(x) = \sum_{\alpha \in \Lambda_T(d,N)} c_\alpha x^\alpha, \, x \in \mathbb{R}$. Analogously, we define the Banach spaces
$\mathcal{P}_{\leq d}(\ell_\infty^N)$
and
$\mathcal{T}_{\leq d}(\ell_\infty^N)$.

Obviously, for each  $f\colon \{-1, +1\}^N \rightarrow \mathbb{R}$ there is a~unique tetrahedral polynomial $P_f : \mathbb{R}^N\rightarrow \mathbb{R}$ for which the  following diagram
commutes:
\begin{equation*} \label{Bild}
\begin{tikzcd}[column sep=small]
\{-1, +1\}^N \arrow[hookrightarrow]{rr}{} \arrow[swap]{dr}{f}& & \,\,\mathbb{R}^{N} \,\,, \arrow{dl}{P_{f}}\\
& \mathbb{R} &
\end{tikzcd}
\end{equation*}
and in this case
\begin{equation*} \label{affine}
\| f\|_{\infty} := \sup_{x \in \{-1, +1\}^N}{|f(x)|} = \sup_{x \in \{-1, +1\}^N}{|P_{f}(x)|} = \sup_{x \in [-1,1]^{N}}{|P_{f}(x)|} =: \| P_{f}\|_{\infty}.
\end{equation*}
Moreover,  each subset  $S \subset [N]$  may be identified with a tetrahedral multi index
$\alpha^S \in  \mathbb{N}_0^N$ given by $\alpha^S(k) = 1, k \in S$ and $\alpha^S(k) = 0, k \notin S$.
Conversely, every
tetrahedral multi index
$\alpha \in  \mathbb{N}_0^N$ defines the subset $S = \text{supp} \,\alpha \subset [N] $.
We write
\[
\Lambda(\mathcal{S}) = \{ \alpha^S \in  \mathbb{N}_0^N \colon S \in \mathcal{S} \}\,,
\]
and $\mathcal{P}_{\Lambda(\mathcal{S})}(\ell_\infty^N (\mathbb{R}))$ for the Banach space of all polynomials on $\ell_\infty^N$,
which are generated by functions  $f \in\mathcal{B}^{N}_{\mathcal{S}}$.
  This all together leads to the  isometric identity
 \begin{equation}\label{identities}
   \mathcal{B}^{N}_{\mathcal{S}} = \mathcal{P}_{\Lambda(\mathcal{S})}(\ell_\infty^N )\,, \,\,\,\,\,\,\,\,f \mapsto P_f\,.
 \end{equation}
 In view of this identification, it from time to time  is  convenient  to use the usual monomial notation, that is, for $S\subset [N]$ we identify the Boolean function $\chi_S$ with $x^{\alpha_S}$.

\subsection{Sidon, unconditional basis and Gordon--Lewis constants}

\text{}

$\bullet$ Recall that the unconditional basis constant of a basis $(e_i)_{i \in I}$ of a Banach space $X$ is given by
the infimum over all $K > 0$ such that for any finitely supported family $(\alpha_i)_{i \in I}$ of scalars and for any finitely supported family $(\varepsilon_i)_{i \in I}$ with  $\varepsilon_i \in  \{-1,+1\}, \,  i \in I$ we have
\begin{equation}\label{unconditionality}
\Big\Vert  \sum_{i \in I} \varepsilon_i \alpha_i e_i \Big\Vert \leq K \Big\Vert \sum_{i \in I} \alpha_i e_i \Big\Vert\,.
\end{equation}
We denote the unconditional basis constant of $(e_i)_{i\in I}$ by
$\boldsymbol{\chi}((e_i)_{i\in I})= \boldsymbol{\chi}((e_i)_{i\in I}; X)$.
  We also write  $\boldsymbol{\chi}((e_i)_{i \in I})= +\infty$, whenever $(e_i)_{i\in I}$ is not unconditional, and say
that $(e_i)_{i\in I}$ is a~$1$-unconditional basis, whenever $\boldsymbol{\chi}((e_i)_{i\in I}) =1$.
The unconditional basis constant
$\boldsymbol{\chi}(X)$
of $X$ is defined to be the infimum of $\boldsymbol{\chi}((e_i)_{i\in I})$ taken over all possible unconditional bases
$(e_i)_{i\in I}$ of~$X$.

It should be noted that, from the Banach space point of view, the Sidon constant  $\sid(\mathcal{B}^N_{\mathcal{S} })$ is nothing else than the unconditional basis constant of the Walsh functions $(\chi_S)_{S \in \mathcal{S} }$, that is   $$\boldsymbol{\sid}(\mathcal{B}^N_{\mathcal{S} }) = \boldsymbol{\chi}((\chi_S)_{S \in \mathcal{S} })\,.$$

$\bullet$ Given Banach spaces $X$, $Y$ and $1 \leq p \leq \infty$, an  operator $u\in \mathcal{L}(X,Y)$ is said to be $p$-factorable
whenever  there exist a~measure space $(\Omega, \Sigma, \mu)$ and operators $v \in  \mathcal{L}(X, L_p(\mu))$,
$w \in \mathcal{L}(L_p(\mu), Y^{\ast\ast})$, satisfying the following factorization
$
  \kappa_Yu\colon X \stackrel{v} \longrightarrow L_p(\mu) \stackrel{w} \longrightarrow Y^{\ast\ast}\,;
$
here, as usual, $\kappa_{Y}\colon Y \to Y^{\ast\ast}$ is the canonical embedding. In this case the $\gamma_p$-norm of the $p$-factorable
operator $u$ is given by
\begin{equation*}\label{fac-p}
 \gamma_p(u) = \inf \|v\| \|w\|\,,
\end{equation*}
where the infimum is taken over all possible factorizations.
We are  mainly interested  in the  norms  $\gamma_p$ for operators acting between  finite dimensional Banach spaces $X$ and $Y$. In
this case, the infimum in \eqref{fac-p} is realized  considering all possible factorizations of the more simple form
\begin{equation*}
\begin{tikzcd}
X  \arrow[rd, "v"']  \arrow[rr, "u"] &  & Y \,\,,\\
& \ell_p^n   \arrow[ru, "w"'] &
\end{tikzcd}
\end{equation*}
where $n$ is arbitrary.

$\bullet$ An operator $u\in \mathcal{L}(X,Y)$ is said to be $1$-summing  if there is a~constant $C>0$ such that
for each choice of finitely many $x_1, \ldots, x_N\in X$ one has
\[
\sum_{j=1}^N \|ux_j\|_Y  \leq C
\sup\Big\{\sum_{j=1}^{N} |x^\ast(x_j)|\,:\,\, \|x^\ast\|_{X^\ast} \leq 1\Big\}\,.
\]
By $\pi_{1}(u\colon X\to Y)$  we denote the least such $C$ satisfying this inequality.

$\bullet$ A Banach space $X$ has the Gordon--Lewis property if every $1$-summing operator $u\colon X \to \ell_2$ is $1$-factorable. In this case, there is a~constant $C>0$ such that for all  1-summing operators $u \colon X \to \ell_2$
\begin{equation*}
\label{eq: def G-L}
\gamma_1(u)\leq C\,\pi_1(u)\,,
\end{equation*}
and the best such $C$ is called the Gordon--Lewis constant of $X$ and denoted by $\boldsymbol{g\!l} (X)$.

A fundamental tool for the study of unconditionality in Banach spaces is the  Gordon--Lewis inequality from \cite{gordon1974absolutely}
(see also \cite[17.7]{diestel1995absolutely} or \cite[Proposition~21.13]{defant2019libro}): For every Banach space $X$ with an unconditional basis~$(e_i)_{i \in I}$ one has
\begin{equation} \label{gl-inequality}
\boldsymbol{g\!l}(X)\leq  \boldsymbol{\chi}(X) \leq \boldsymbol{\chi}( (e_i)_{i \in I}).
\end{equation}
In contrast to the unconditional basis constant, the Gordon--Lewis constant has the useful (ideal) property that
\begin{equation} \label{niceprop}
\boldsymbol{g\!l}(X)\leq \|u\|  \|v\|\,\,\boldsymbol{g\!l}(Y)\,,
\end{equation}
whenever $\text{id}_{X} = uv$ for appropriate operators $u: X \to Y$ and $v: Y \to X$.

\subsection{Projection constants} \label{projection constants}
General bounds for projection constants  of various finite dimensional Banach spaces were studied by many authors
(see again \eqref{definition} for the definition).

The projection constant of a~Banach space $X$ can be formulated in terms of the  $\infty$-factorization norm of the identity
operator $\id_X$.  More precisely, if $X$ is a~Banach space and $X_0$
is any subspace of some $L_{\infty}(\mu)$ isometric to $X$, then
\begin{align} \label{gammainfty}
\boldsymbol{\lambda}(X)= \gamma_\infty(\id_X)= \boldsymbol{\lambda}(X_0, L_{\infty}(\mu)).
\end{align}

Recall the  following fundamental
estimate due to Kadets-Snobar \cite{kadecsnobar}: For every $n$-dimen\-sional Banach space $X_n$ one has
\begin{equation} \label{kadets1}
\boldsymbol{\lambda}(X_n) \leq \sqrt{n}\,.
\end{equation}
In contrast, K\"onig and Lewis \cite{koniglewis} showed that for any Banach space $X_n$ of dimension $n \ge 2$ the strict
inequality $\boldsymbol{\lambda}(X_n) < \sqrt{n}$ holds, and this estimate was later improved by Lewis \cite{lewis}.

The exact values of $\boldsymbol{\lambda}(\ell_2^n)$ and $\boldsymbol{\lambda}(\ell_1^n)$ were computed by Gr\"unbaum
\cite{grunbaum} and Rutovitz \cite{rutovitz}: In the complex case
\begin{equation*}\label{grunbuschC-A}
  \boldsymbol{\lambda}\big(\ell_2^n(\mathbb{C})\big)  = n \int_{\mathbb{S}_{n-1}^{\mathbb{C}}} |x_1|\,d\sigma(x)
= \frac{\sqrt{\pi}}{2}   \frac{n!}{\Gamma(n + \frac{1}{2})}\,,
\end{equation*}
where $d\sigma$ stands for the normalized surface measure on the sphere  $\mathbb{S}_{n-1}^{\mathbb{C}}$
in $\mathbb{C}^n$, and
\begin{equation*}\label{grunbuschC-B}
\boldsymbol{\lambda}\big(\ell_1^n(\mathbb{C})\big)  = \int_{\mathbb{T}^n} \Big|\sum_{k=1}^{n} z_k\Big|\, dz
= \int_{0}^{\infty} \frac{1 -J_0(t)^n}{t^2} dt\,,
\end{equation*}
where $dz$ denotes the normalized Lebesgue measure on the distinguished boundary $\mathbb{T}^n$ in $\mathbb{C}^n$
and $J_0$ is the zero Bessel function defined by
$
J_0(t) = \frac{1}{2\pi} \int_{0}^{\infty} \cos( t \cos \varphi) d \varphi\,.
$
The corresponding real constants are different:
\begin{equation*}
 \boldsymbol{\lambda}\big(\ell_2^n(\mathbb{R})\big)  =  n \int_{\mathbb{S}_{n-1}^{\mathbb{R}}} |x_1|\,d\sigma
= \frac{2}{\sqrt{\pi}}   \frac{\Gamma(\frac{n+2}{2})}{\Gamma(\frac{n+1}{2})}
\end{equation*}
and
\begin{align}\label{grunbuschR}
\boldsymbol{\lambda}\big(\ell_1^n(\mathbb{R})
\big)  =
\begin{cases}
\boldsymbol{\lambda}\big(\ell_2^n(\mathbb{R})\big),  &  \text{$n$ odd}\\[2mm]
\boldsymbol{\lambda}(\ell_2^{n-1}(\mathbb{R})),  &  \text{$n$ even}\,.\\[2mm]
\end{cases}
\end{align}
Additionally, Gordon \cite{gordon} and Garling-Gordon \cite{garlinggordon} determined the asymptotic growth of
$\boldsymbol{\lambda}\big(\ell_p^n\big)$ for $1<p<\infty$ with $p\neq 1,2,\infty$ showing that
\begin{equation*} \label{gordongarling}
\boldsymbol{\lambda}\big(\ell^n_{p}\big) \sim n^{\min\big\{\frac{1}{2}, \frac{1}{p} \big\}}\,,
\end{equation*}
and K\"onig, Sch\"utt and Tomczak-Jagermann \cite{konig1999projection} proved that for  $1 \le p \le 2$
\begin{equation}\label{koenigschuetttomczak}
  \lim_{n\to \infty} \frac{\boldsymbol{\lambda}\big(\ell_p^n\big)}{\sqrt{n}} = \gamma\,,
\end{equation}
where $\gamma = \sqrt{\frac{2}{\pi}}$ in the real  and $\gamma= \frac{\sqrt{\pi}}{2}$ in the complex case.
For an extensive treatment on all of this and more see
\cite{tomczak1989banach}.

\section{Integral formula}
The following integral formula for the projection constant of  $\mathcal{B}_\mathcal{S} ^{N}$, where $\mathcal{S}$ is an arbitrary family of subsets in $[N]$, is fundamental for our purposes.

\begin{theorem} \label{bool-int}
For each family  $\mathcal{S}$ of subsets in $[N]$
\[
\boldsymbol{\lambda}\big(\mathcal{B}_\mathcal{S} ^{N}\big) = \mathbb{E}\big[    \big| \sum_{S \in \mathcal{S} } \chi_S \big| \big]\,.
\]
\end{theorem}

This immediately follows from a result regarding arbitrary compact abelian groups as presented in \cite[Theorem~2.1]{defant2024projection}, which finds its inspiration and roots in Rudin's work \cite{rudin1962projections} (also see \cite[Theorem III.B.13]{wojtaszczyk1996banach}).
In fact,  for
 a compact abelian group  $G$ (with Haar measure $\mathrm{m}$) and  a~finite set $E:=\{\gamma_1,\ldots, \gamma_N\}\subset \widehat{G}$   of
characters, we denote by $\text{Trig}_E(G)$ the finite dimensional Banach space of all trigonometric polynomials formed by the span of $E$  in $C(G)$. Then the projection $P\colon C(G) \to C(G)$, given by $Pf = \sum_{j=1}^N \widehat{f}(\gamma_j) \gamma_j$ for
all $f\in C(G)$, is a~minimal projection onto $\text{Trig}_E(G)$ and
\[
\boldsymbol{\lambda}\big(\text{Trig}_E(G)\big) = \big\|P\colon C(G)\to C(G)\big\|
= \int_G \Big|\sum_{j=1}^N \gamma_j(x)\Big|\,d\mathrm{m}(x)\,.
\]
Taking for $G$ the $N$-dimensional Boolean cube $\{-1, +1\}^N$ and recalling
that all its characters are given by the functions $\chi_S, \, S \subset [N]$, we see that Theorem~\ref{bool-int}
indeed is a very special case.

Clearly, by the Kadets-Snobar theorem (recall again \eqref{kadets1}) one has
\begin{equation}\label{KSKS}
1 \,\,\leq  \,\,\boldsymbol{\lambda}\big(\mathcal{B}^N_\mathcal{S} \big)  \,\,\leq  \,\, \sqrt{|\mathcal{S} |}\,.
\end{equation}
Note that this estimate is also a straight forward consequence of  Theorem~\ref{bool-int} and  the orthogonality of the
Fourier--Walsh functions $\chi_S$, since  both imply
\begin{align*}
\boldsymbol{\lambda}\big(\mathcal{B}_\mathcal{S} ^{N}\big) =\mathbb{E}\Big[\Big| \sum_{S \in \mathcal{S} } \chi_S \Big| \Big]\le
\Big( \mathbb{E}\Big[\Big| \sum_{S \in \mathcal{S} } \chi_S \big|^2 \Big]\Big)^{1/2} =  \sqrt{|\mathcal{S} |}\,.
\end{align*}
We  show that the estimates of \eqref{KSKS},  the upper as well as the lower one, are  attained. Indeed, the two possible
extreme cases for  \eqref{KSKS} are as follows: For the lower bound we take the family $\mathcal{S}$ consisting of all possible  subsets
of $[N]$, and for the upper bound any family  $\mathcal{S}$ of  one-point sets only.
In the first case the identity from \eqref{identitiesI} gives that for all $N$
\begin{equation}\label{caseA}
  \boldsymbol{\lambda}\big(\mathcal{B}^{N}
\big) = 1\,,
\end{equation}
and in the second case  by Kintchine's inequality (as proved by Szarek in \cite{szarek1976best})
\begin{equation}\label{caseB}
  \frac{1}{\sqrt{2}}  \sqrt{|\mathcal{S}|}  \leq \  \mathbb{E}\Big[ \big|\sum_{\{x\} \in \mathcal{S}}   \chi_{\{x\}}  \big|  \Big]\leq \sqrt{|\mathcal{S}|}\,.
\end{equation}
For the latter case of  one-point sets we have a more precise formula.
Indeed, for each family $\mathcal{S}$ of one-point sets of $[N]$ by Theorem~\ref{bool-int} and an integral form of Rademacher averages proved by Haagerup in \cite{haagerup1981best} one has
\begin{equation} \label{caseBB}
\boldsymbol{\lambda}\big(\mathcal{B}_\mathcal{S} ^{N}\big)= \frac{2}{\pi}\,
\int_0^{\infty} t^{-2}\bigg(1 - \prod_{k=1}^{|\mathcal{S}|} \cos t\bigg)\,dt\,.
\end{equation}
Moreover, if  $(\mathcal{S}_N)_{N\in \mathbb N}$ is a sequence of support sets which are finite and consists only of one-point sets  in $\mathbb N$ such that the cardinality 
 $|\mathcal{S}_N| \to \infty$ as  $N \to \infty$, then  
\begin{equation} \label{caseBBB}
\lim_{N \to \infty}  \frac{\boldsymbol{\lambda}\big(\mathcal{B}_{\mathcal{S}_N}^{N}\big)}{\sqrt{|\mathcal{S}_N|}} = \sqrt{\frac{2}{\pi}}\,.
\end{equation}
This follows by a  standard duality argument, since  the mapping
\[
\mathcal{B}_{\mathcal{S}_N}^{N} \ni \sum_{\{i_k\}\in \mathcal{S}_N } \alpha_k \chi_{\{i_k\}} \mapsto (\alpha_k)_{k=1}^N
\]
is a~linear isometric isomorphism of $\mathcal{B}_{\mathcal{S}_N}^{N}$ onto $\ell_1^{|\mathcal{S}_N|}$, and then  the conclusion
follows from the case $p=1$ in~\eqref{koenigschuetttomczak}.

The  following consequences of Theorem~\ref{bool-int} collect a few extensions of~\eqref{caseB}.
The substitute for Kintchine's inequality in the case $d=1$ we need, is the following hypercontractivity  estimate for homogeneous functions on Boolean
cubes  due to Ivanisvili and Tkocz \cite[Theorem 2]{ivanisvili2019comparison} which shows that, for $d>1$ and every $f\in  \mathcal{B}_{=d}^{N}$,
\begin{equation}\label{homogeneousweissler}
\|f\|_{L_2(\{-1,+1\}^N)} \leq  e^{\frac{d}{2}} \|f\|_{L_1(\{-1,+1\}^N)}\,.
\end{equation}
More generally,  in \cite[Theorem 13]{eskenazis2020polynomial} Eskenazis and Ivanisvili showed that for every $f \in \mathcal{B}_{\leq d}^{N}$
\begin{equation}\label{weissler}
\|f\|_{L_2(\{-1,+1\}^N)} \leq  (2.69076)^d \|f\|_{L_1(\{-1,+1\}^N)}\,.
\end{equation}

The previous two inequalities are recent improvements of the classical hypercontractive Bonami-Weissler inequality for the Boolean cube 
(see e.g., \cite[Theorem 9.22]{o2014analysis}).

Combining Theorem~\ref{bool-int}  with  \eqref{weissler} and \eqref{homogeneousweissler} gives the following extension of ~\eqref{caseB}.

\begin{corollary}\label{kiel}
Let   $\mathcal{S}  \subset 2^{[N]}$. 
If $|S|= d$
for all $S \in \mathcal{S}$, then 
\[
\frac{1}{e^{\frac{d}{2}}} \sqrt{|\mathcal{S} |}\,\,\leq\,\,\boldsymbol{\lambda}\big(\mathcal{B}_{\mathcal{S}}^{N}\big) \,\,\leq\,\, \,\sqrt{|\mathcal{S}|}\,,\]
and if  $|S|\leq d$
for all $S \in \mathcal{S}$, then
\[
\frac{1}{(2.69076)^d} \sqrt{|\mathcal{S} |}\,\,\leq\,\,\boldsymbol{\lambda}\big(\mathcal{B}_{\mathcal{S}}^{N}\big) \,\,\leq\,\, \,\sqrt{|\mathcal{S}|}\,.\]
\end{corollary}
Calculating cardinalities, yields more concrete estimates for
$\mathcal{B}_{ = d}^{N}$ and~$\mathcal{B}_{\leq d}^{N}$:
For  $1 \leq d \leq N$ 
\begin{equation} \label{peskowA}
\frac{1}{e^{\frac{d}{2}}}  \,  \binom{N}{d}^{\frac{1}{2}}
\, \, \leq \,\, \boldsymbol{\lambda}\big(\mathcal{B}_{=d}^{N}\big) \,\,\leq\,\,
\binom{N}{d}^{\frac{1}{2}}
\end{equation}
\begin{equation} \label{peskowB}
\frac{1}{(2.69076)^d}\,  \left( \sum_{k=0}^d \binom{N}{k}\right)^{\frac{1}{2}}
\, \, \leq \,\, \boldsymbol{\lambda}\big(\mathcal{B}_{\le d}^N\big) \,\,\leq\,\, \left( \sum_{k=0}^d \binom{N}{k}\right)^{\frac{1}{2}}\,,
\end{equation}
and as a consequence
\begin{equation} \label{mariupol}    
\frac{1}{e^{\frac{d}{2}}}
\bigg(\frac{N}{d}\bigg)^{\frac{d}{2}}\,\leq\,\boldsymbol{\lambda}\big(\mathcal{B}^N_{= d}\big) \,\leq\,
e^{\frac{d}{2}}\bigg(\frac{N}{d}\bigg)^{\frac{d}{2}}
\end{equation}
\begin{equation} \label{mariupol2}    
\frac{1}{(2.69076)^{d}}
\bigg(\frac{N}{d}\bigg)^{\frac{d}{2}}\,\leq\,\boldsymbol{\lambda}\big(\mathcal{B}^N_{\leq d}\big) \,\leq\,
e^{\frac{d}{2}}\bigg(\frac{N}{d}\bigg)^{\frac{d}{2}}\,.
\end{equation}
Indeed, the preceding two estimates  follow  immediately from~\eqref{peskowA} 
and~\eqref{peskowB}  taking into account the  elementary estimates
\begin{align} \label{ukraineA}
\Big(\frac{N}{d}\Big)^d \leq \binom{N}{d}\,,
\end{align}
\begin{equation} \label{ukraineAA}
\binom{N}{d}
\leq
\sum_{k=0}^d \binom{N}{k} \leq \sum_{k=0}^d  \frac{N^k}{k!} = \sum_{k=0}^d  \frac{d^k}{k!} \Big(\frac{N}{d}\Big)^k
\leq e^d \Big(\frac{N}{d}\Big)^d\,.
\end{equation}

Note that applying a~remarkable formula due to~McKay \cite{mckay1989littlewood}, we have (see also \cite[Lemma 5.7]{defant2018bohrBoolean}):
For each $N \in \N$ and each $0 \leq \alpha <N$
with $N - \alpha$ being an odd integer, there exists $0 \leq c_{\alpha,N} \leq \sqrt{\pi/2}$ such that
\[
\sum_{k \leq \frac{N - \alpha - 1}{2}}{\binom{N}{k}}
= \sqrt{N}\,\binom{N-1}{\frac{N - \alpha -1}{2}}\,Y\left( \frac{\alpha + 1}{\sqrt{N}} \right)
\,\exp{\left(  \frac{c_{\alpha, N}}{\sqrt{N}} \right)}\,,
\]
where $Y$ is given by
\[
Y(x) = e^{\frac{x^{2}}{2}} \int_{x}^{\infty}e^{-\frac{t^{2}}{2}}\,dt, \quad\, x\geq 0\,.
\]
In particular,  taking $\alpha =0$, we obtain a~nice asymptotic formula for $\sum_{k=0}^{d}\binom{N}{k}$,
whenever $N$ is odd and $d = \frac{N-1}{2}$.

It worth noting that the following estimates hold (see \cite[Proposition 3]{szarek1999nonsymmetric})
\[
\frac{2}{x + (x^{2} + 4)^{1/2}} \leq Y(x) \leq \frac{4}{3x + (x^{2} + 8)^{1/2}}, \quad\, x\geq 0\,.
\]
We conclude with an observation showing  that the upper bound in \eqref{ukraineAA} can be improved as follows whenever
$2d-1 < N$:
\begin{equation} \label{Desigforo}
\sum_{k=0}^d \binom{N}{k} \le {N \choose d} \frac{N-(d-1)}{N-(2d-1)}\,.
\end{equation}
 In fact, we have
\[
\frac{1}{{N \choose d}}\Bigg({N \choose d} + {N \choose d-1} +  {N \choose d-2} +\,\cdots\, + 1\Bigg)
= 1 + \frac{d}{N-d+1} + \frac{d(d-1)}{ (N-d+1)(N-d+2)} + \,\cdots\,.
\]
Thus bounding  the right-hand side from above by the geometric series
\[1 + \frac{d}{N-d+1} + \left( \frac{d} {N-d+1} \right)^2
+ \, \cdots\,=\, \frac{N-(d-1)}{N - (2d-1)}\,,
\]
we get the required estimate \eqref{Desigforo}.

In view of the preceding estimates for the  projection constants of
$\mathcal{B}_{=d}^{N}$ and $\mathcal{B}_{\le d}^{N}$, we add another useful  result comparing both constants.

\begin{proposition}
For every $N,d \in \mathbb{N}$ with $d \leq N$
\[
\boldsymbol{\lambda}
\big(\mathcal{B}_{=d}^{N}\big) \le (1+\sqrt{2})^d  \boldsymbol{\lambda}
\big(\mathcal{B}_{\le d}^{N}\big)\,.
\]
\end{proposition}

\begin{proof}
Indeed, this follows   from an  important  fact proved by Klimek in  \cite{klimek1995metrics}
(see also \cite[Lemma1,(iv)]{defant2019fourier}): If $f \in \mathcal{B}_{\leq d}^{N} $ and
$f_k = \sum_{|S|=k} \hat{f}(S) \chi_S$ is the $k$-homogeneous part of $f$  for $0 \leq k \leq d$,
then
\begin{equation}\label{klimek}
\|f_k\|_\infty \leq (1+\sqrt{2})^d \|f\|_\infty \,. \qedhere
\end{equation}
\end{proof}

We finish with two more remarks on  Theorem~\ref{bool-int}, which have a number theoretical flavour.
The first one is
\begin{equation} \label{prime1}
\lim_{N \to \infty}\frac{\boldsymbol{\lambda}\big(\mathcal{B}^N_{\mathcal{P}_N}\big)}{ \sqrt{\frac{N}{\log N}}} = \sqrt{\frac{2}{\pi}}\,,
\end{equation}
where $\mathcal{P}_N$ stands the family of one-point sets $\{p\}$ generated by all  primes~$p\leq~N$.
Of course this is an immediate consequence of the prime number theorem and~\eqref{caseBBB}
applied to the sequence $(p_k)$ of all primes.

For the second remark  we denote by $\mathcal{P}_N^{sf}$ the family of all
finite subsets $A$ of primes in  $[N]$ such that $n = \prod_{p \in A} p \leq N$ (the prime number decomposition of $n$
is square-free). Observe, that since every
$n \in [N]$ has a~unique prime number decomposition, there is a one to one correspondence between the set of all square-free numbers $n \in [N]$ and $\mathcal{P}_N^{sf}$.

Based on a~recent deep result of Harper from \cite{Harper}, for large integers $N$ with universal constants 
\begin{equation} \label{prime2}
\boldsymbol{\lambda}\Big(\mathcal{B}^N_{\mathcal{P}_N^{sf}}\Big)
\,\sim\,\frac{\sqrt{N}}{(\log \log N)^{\frac{1}{4}}}\,.
\end{equation}
For the proof  note that all projections 
$\chi_{\{k\}}\colon \{-1,+1\}^{\mathbb{N}} \to \mathbb{R},\, k\in \mathbb{N}$ form a sequence of  independent
 Rademacher random variables. Moreover, every square-free number $n \in [N]$ defines a random variable
 $f_n \colon \{-1,+1\}^{\mathbb{N}} \to \mathbb{R}$
given by
\[
f_n = \prod_{\substack{p|n\\ p \,\text{prime}}} \chi_{\{p\}}\,.
\]
 Then Harper's result mentioned above states that
\[
\mathbb{E}  \Big| \sum_{\substack{1 \leq n \leq N\\  \text{$n$ square-free}}} f_n\Big|
\sim  \frac{\sqrt{N}}{(\log \log N)^{\frac{1}{4}}} \,.
\]
But since the one to one correspondence between the square-free numbers $n \in [N]$
and the family of sets $A \in \mathcal{P}_N^{sf}$ identifies the random variables
$f_n$ and $\chi_A$, we see that
\[
\mathbb{E}  \Big| \sum_{\substack{1 \leq n \leq N\\  \text{$n$ square-free}}}  f_n\Big|
= \mathbb{E}  \Big| \sum _{A \in \mathcal{S}} \chi_ A\Big|\,.
\]
 Consequently, Theorem~\ref{bool-int}  finishes the argument for \eqref{prime2}.

\section{The limit case}
The spaces $\mathcal{B}_{=1}^N$ and $\ell_1^N(\mathbb{R})$ identify as Banach spaces
whenever we interpret the $N$-tuple $\sum_{k=1}^N x_k e_k$ as the function $\sum_{k=1}^N x_k \chi_{\{k\}}$
on the $N$-dimensional Boolean cube.
Then by the result of Gr\"unbaum mentioned in~\eqref{grunbuschR}  we know that
\[
\lim_{N \to \infty} \frac{\boldsymbol{\lambda}\big(\mathcal{B}_{=1}^N\big)}{\sqrt{N}}  =
\lim_{N \to \infty} \frac{\boldsymbol{\lambda}\big(\ell_1^N(\mathbb{R})\big)}{\sqrt{N}}
= \sqrt{\frac{2}{\pi }}\,.
 \]
  In the following we show a procedure that allows to extend this result to  $\mathcal{B}_{=d}^N$ and $\mathcal{B}_{\leq d}^N$,
  where the degree $d$ is arbitrary.
Our main result here is as follows.

\begin{theorem}\label{thm: limit bool homog}
For each $d \in \NN$,
\begin{equation}\label{limform}
\dis\lim_{N \to \infty} \frac{\boldsymbol{\lambda}(\mathcal{B}_{=d}^N)}{N^{d/2}} = \frac{1}{\sqrt{2 \pi}} \int_{-\infty}^{\infty} |P_d(t)|e^{-\frac{t^2}{2}} dt\,,
\end{equation}
where  $P_d(t) = \frac{t^d}{d!} - \dis\sum_{k = 1}^{\floor{d/2}} \frac{1}{k!2^k} P_{d-2k}(t) $
with $P_0(t) = 1$, $P_1(t) = t$. Moreover,
$$
\dis\lim_{N \to \infty} \frac{\boldsymbol{\lambda}(\mathcal{B}_{=d}^N)}{N^{d/2}} = \dis\lim_{N \to \infty} \frac{\boldsymbol{\lambda}(\mathcal{B}_{\le d}^N)}{N^{d/2}}\,.
$$
\end{theorem}

\smallskip

The following observation is important, since it helps  to understand the preceding result  in a~larger context:
For each $d \in \mathbb{N}_0$, we have that 
\begin{equation} \label{eqPdandHermite}
P_d=\frac{h_d}{d!}\,,
\end{equation}
where $h_n$ for each $n \in \mathbb{N}_0$ denotes the $n$-th (probabilist's) Hermite polynomial  given by
\[
h_n(t) := (-1)^n e^{\frac{t^2}{2}} \frac{d^n}{dt^n} e^{-\frac{t^2}{2}}, \quad\, t \in \mathbb{R}\,.
\]
In order to prove this, note first that for $d \in \{0, 1\}$  this equality holds trivially. For arbitrary $d$'s  we use induction. Suppose that $P_k=\frac{h_k}{k!}$ for every $0 \leq k \leq d-1$, and let us show that $P_{d}=\frac{h_{d}}{{d}!}$.
By the so-called 'inverse explicit expression for Hermite polynomials' we know that for all $t\in \mathbb{R}$
\[
t^d = d! \dis\sum_{k = 0}^{\floor{d/2}} \frac{1}{k!\,2^k}\,\frac{h_{d-2k}(t)}{(d-2k)!} = d! \dis\sum_{k = 1}^{\floor{d/2}} \frac{1}{k!\,2^k}\,\frac{h_{d-2k}(t)}{(d-2k)!} + h_{d}(t) \,.
\]
Thus, by the inductive hypothesis,
\[
\frac{h_d(t)}{d!} = \frac{t^d}{d!} -  \dis\sum_{k = 1}^{\floor{d/2}} \frac{1}{k!\,2^k}\,\frac{h_{d-2k}(t)}{(d-2k)!} = \frac{t^d}{d!} - \dis\sum_{k = 1}^{\floor{d/2}} \frac{1}{k!2^k} P_{d-2k}(t)  =  P_d(t)\,.
\]
Note that  from~\eqref{peskowA} for
 $\mathcal{S}= \{S \subset [N]:\, \text{card}(S) = d\}$  we have
\[
\frac{1}{e^{\frac{d}{2}}}
\,\,\leq\,\,
\liminf_{N \to \infty} \frac{ \boldsymbol{\lambda}\big(\mathcal{B}_{\mathcal{S}}^N\big)}{\sqrt{\dim( \mathcal{B}_{\mathcal{S}}^N)}}
\,\,\leq\,\,
\limsup_{N \to \infty} \frac{ \boldsymbol{\lambda}\big(\mathcal{B}_{\mathcal{S}}^N\big)}{\sqrt{\dim( \mathcal{B}_{\mathcal{S}}^N)}}
\,\,\leq\,\,
1\,,
\]
and in  the case $\mathcal{S}= \{S\subset [N]: \, \text{card}(S) \leq d\}$ the constant $e^{-\frac{d}{2}}$ has to be changed by  
$(2.69076)^{-d}$.
The following result is a  considerable improvement.

\begin{corollary}   
For $\mathcal{S}= \{S\subset [N]:\, \text{card}(S) = d\}$ or $\mathcal{S}= \{S\subset [N]: \, \text{card}(S) \leq d\}$
\begin{equation*} 
\lim_{N \to \infty} \frac{ \boldsymbol{\lambda}\big(\mathcal{B}_{\mathcal{S}}^N\big)}{\sqrt{\dim( \mathcal{B}_{\mathcal{S}}^N)}}  
\,\,=\,\,
\frac{1}{\sqrt{2 \pi }} \int_{-\infty}^{\infty} \frac{|h_d(t)|}{\sqrt{d!}}e^{-\frac{t^2}{2}}\,dt
\,= \,
\frac{2^{7/4}}{\pi^{5/4}} \frac{1}{d^{1/4} }\,\bigg(1 + O\Big(\frac{1}{d^2}\Big)\bigg)\,.
\end{equation*}
\end{corollary}

Indeed, this follows from  Theorem~\ref{thm: limit bool homog}, equation \eqref{eqPdandHermite}
and  a result of Larsson-Cohn \cite[Remark 2.6 and 3.2]{larsson2002p}, which says that
\begin{equation*}
\frac{1}{\sqrt{2 \pi }} \int_{-\infty}^{\infty} |h_d(t)|e^{-\frac{t^2}{2}}\,dt
\,= \,
\frac{2^{7/4}}{\pi^{5/4}} \frac{\sqrt{d!}}{d^{1/4} }\,\bigg(1 + O\Big(\frac{1}{d^2}\Big)\bigg)\,.
\end{equation*}  
Now, if  $\mathcal{S}= \{S: \, \text{card}(S) = d\}$, then 
$\dim( \mathcal{B}_{\mathcal{S}}^N)= \sqrt{\binom{N}{d}}$, and if
$\mathcal{S}= \{S: \, \text{card}(S) \leq d\}$, then in the case $2d-1 < N$ (so in particular for large $N$), we have  (see again~\eqref{Desigforo})  
\begin{align*} 
{N \choose d} \leq \dim( \mathcal{B}_{\mathcal{S}}^N)
 = \sum_{k=0}^d \binom{N}{k} \le {N \choose d} \frac{N-(d-1)}{N-(2d-1)}\,.
\end{align*}
Since $\lim_{N\to \infty} \sqrt{{N \choose d}}{\big/}N^{d/2} = 1/\sqrt{d!}$\,, we in both considered cases have  
\[ 
\lim_{N\to \infty} \frac{\sqrt{\dim(\mathcal{B}_{\mathcal{S}}^N)}}{N^{d/2}} = \frac{1}{\sqrt{d!}}\,,
\]
which gives the required statement.

To show a few examples, with the use of a computational platform, we get the following limits:
\begin{align*}
    \dis\lim_{N \to \infty} \frac{\boldsymbol{\lambda}(\Bb_{=2}^N) }{\sqrt{\dim(\mathcal{B}_{=2}^N)}} & = \dis\lim_{N \to \infty} \frac{\boldsymbol{\lambda}(\Bb_{\le 2}^N) }{\sqrt{{N \choose 2}}} = \frac{1}{\sqrt{2 \pi}} \int_{-\infty}^{\infty} \left| \frac{t^2-1}{\sqrt{2}} \right| e^{-\frac{t^2}{2}} dt = {\frac{2^{7/4}}{\sqrt e \sqrt{2\pi}}}\approx \frac{0.814}{ 2^{1/4}} \\[1ex]
    \dis\lim_{N \to \infty} \frac{\boldsymbol{\lambda}(\Bb_{=3}^N) }{\sqrt{\dim(\mathcal{B}_{=3}^N)}} & = \dis\lim_{N \to \infty} \frac{\boldsymbol{\lambda}(\Bb_{\le 3}^N) }{{\sqrt{{N \choose 3}}}} = \frac{1}{\sqrt{2 \pi}} \int_{-\infty}^{\infty} \left| \frac{t^3-3t}{\sqrt 6} \right| e^{-\frac{t^2}{2}} dt =\frac{1}{3\sqrt{2 \pi} } \left(1 + \frac{4}{e^{3/2}}\right) \approx \frac{0.811}{3^{1/4}} \\[1ex]
        \dis\lim_{N \to \infty} \frac{\boldsymbol{\lambda}(\Bb_{=4}^N) }{\sqrt{\dim(\mathcal{B}_{=4}^N)}} & = \dis\lim_{N \to \infty} \frac{\boldsymbol{\lambda}(\Bb_{\le 4}^N) }{{\sqrt{{N \choose 4}}}} = \frac{1}{\sqrt{2 \pi}} \int_{-\infty}^{\infty} \left| \frac{t^4-6t^2+3}{\sqrt {24}} \right| e^{-\frac{t^2}{2}} dt \approx \frac{0.808}{4^{1/4}} \\[1ex]
        \dis\lim_{N \to \infty} \frac{\boldsymbol{\lambda}(\Bb_{=5}^N) }{\sqrt{\dim(\mathcal{B}_{=5}^N)}} & = \dis\lim_{N \to \infty} \frac{\boldsymbol{\lambda}(\Bb_{\le 5}^N) }{{\sqrt{{N \choose 5}}}} = \frac{1}{\sqrt{2 \pi}} \int_{-\infty}^{\infty} \left| \frac{t^5-10t^3+15t}{\sqrt {120}} \right| e^{-\frac{t^2}{2}} dt  \approx \frac{0.807}{5^{1/4}}   \\[1ex]
        \dis\lim_{N \to \infty} \frac{\boldsymbol{\lambda}(\Bb_{=6}^N) }{\sqrt{\dim(\mathcal{B}_{=6}^N)}} & = \dis\lim_{N \to \infty} \frac{\boldsymbol{\lambda}(\Bb_{\le 6}^N) }{{\sqrt{{N \choose 6}}}} = \frac{1}{\sqrt{2 \pi}} \int_{-\infty}^{\infty} \left| \frac{t^6-15t^4+45t^2-15}{\sqrt{ 720}} \right| e^{-\frac{t^2}{2}} dt \approx \frac{0.806}{6^{1/4}}. 
\end{align*}


\smallskip

 For the proof of Theorem~\ref{thm: limit bool homog} we  use a probabilistic point of view, treating  the coordinate functions $(\chi_{\{i\}})_{1 \le i \le N}$ on the Boolean cube  as independent Bernoulli random variables (taking the values $\pm 1$ with equal probability $1/2$);
  for the random variable $\chi_{\{i\}}$ we shortly write $x_i$.
    From this perspective, any Walsh function $\chi_S$  is itself a random variable, being a  product of coordinate functions. Consequently, any function on the Boolean cube may also be seen as a random variable.

 By Theorem~\ref{bool-int} the projection constant of  $\mathcal{B}_{=d}^N$ can be computed as the expectation
 \begin{equation*}
   \mathbb{E} \dis \bigg| \sum_{ |S| = d, S \subset [N]  } \chi_S \bigg|\,.
 \end{equation*}
 Moreover we know from the  central limit theorem that
 \begin{equation}\label{CLT}
   \frac{\sum_{i = 1}^N x_i}{\sqrt{N}} \,\,\longrightarrow \,\,Z
   \,,
 \end{equation}
  where $Z$ is a normal random variable with mean  $0$ and variance~$1$,
  and the convergence is in distribution.
      Based on this,
the main idea of our procedure is to rewrite  the random variable
$  \sum_{ |S| = d, S \subset [N]  } \chi_S $
in a suitable way into another
random variable involving $\sum_{i = 1}^N x_i/\sqrt{N} $, for which we manage to control its mean.

We use the notation $Y_n \overset{D}{\longrightarrow}Y$, whenever a sequence $(Y_n)$ converges in distribution to a random variable $Y$. Additionally to the notion of convergence in distribution, it will be necessary to consider convergence in probability.
We  write $Y_n \overset{P}{\longrightarrow}Y$  if the sequence $(Y_n)$ converges in probability to a~random variable $Y$.  Of course, convergence in probability  implies convergence in distribution  but, in general, the converse is not true. We recall that these two notions of convergence coincide, provided  the limit is a constant.

Moreover, we frequently  need a classical theorem of Slutsky. It states that, given two sequences  $(X_n)_{n }$ and $(Y_n)_{n}$ of random variables  such that $X_n \overset{D}{\longrightarrow} X$ and $Y_n \overset{P}{\longrightarrow} c$ (where $X$ is a random variable and $c \in \mathbb{R}$ a constant), then $X_n + Y_n \overset{D}{\longrightarrow} X + c$ and $X_n  Y_n \overset{D}{\longrightarrow} c X$.

Another result used at several places is  that convergence in distribution is inherited under continuous functions
in the sense that  $f(Y_n) \overset{D}{\longrightarrow} f(Y)$, whenever $Y_n \overset{D}{\longrightarrow} Y$ and $f: \mathbb{R} \to \mathbb{R}$ is continuous.

Finally, recall  that a sequence of random variables $(Y_n)_{n }$ is said to be uniformly integrable, whenever $$\lim_{a \to  \infty} \sup_{n \ge 1} \int_{|Y_n| \ge a } |Y_n| d P = 0\,.$$ A sufficient condition is that $\sup_{n} \mathbb{E}( |Y_n|^{1+\varepsilon}) \le C$ for some $\varepsilon, C>0$, since then
\begin{equation} \label{uni-int}
    \lim_{a \to  \infty} \sup_{n \ge 1} \int_{|Y_n| \ge a } |Y_n| d P \le \lim_{a \to  \infty} \frac{1}{a^\varepsilon} C\,.
\end{equation}
Uniform integrability is of particular importance to us  due to the fact that
 $Y$ is integrable and
\begin{align}\label{thm: unif int + conv in dist implies conv in mean}
\mathbb{E} (Y_n) \to \mathbb{E}(Y)\,,
\end{align}
provided $(Y_n)_{n}$ is a uniformly integrable and $Y_n \overset{D}{\longrightarrow} Y$ (see for example \cite[Theorem 3.5]{billingsley2013convergence}).

\subsection{The $2$-homogeneous case}
To keep our later  calculations
for the proof of
Theorem~\ref{thm: limit bool homog}
more transparent, we deal in detail first with the 2-homogeneous case.
\begin{theorem}\label{thm: cte proy boole 2 hom}
$$\dis\lim_{N \to \infty} \frac{\boldsymbol{\lambda}(\Bb_{=2}^N) }{N} = \sqrt{\frac{2}{\pi e}}$$
\end{theorem}

\begin{proof}
In a first step we expand the square of the Boolean function $x \mapsto \sum_{i=1}^N x_i$ and rearrange the terms using $x_i^2 = 1$, to get
\[
\dis\sum_{1 \le i < j \le N}  x_i x_j = \dis\sum_{1 \le i < j \le N}  x_{\{i,j\}} = \frac{1}{2} \left[\bigg( \sum_{i=1}^N x_i \bigg)^2 - N \right]\,.
\]
By the central limit theorem the sequence of random variables $(Z_N)$ given by
$$
Z_N := \frac{1}{\sqrt{N}}\dis\sum_{i =1}^N x_i $$
converges in distribution to a normal random variable $Z$ with mean $0$ and variance $1$. Since the function $f(x) = \frac{|x^2-1|}{2}$ is continuous, we have
\[
\frac{1}{N} \bigg|
\sum_{1 \le i < j \le N}  x_i x_j \bigg| = \frac{|Z_N^2 -1 |}{2} \overset{D}{\longrightarrow} \frac{|Z^2 -1 |}{2}.
\]
Now note that by the orthogonality of the Fourier--Walsh basis get
\[
\mathbb{E} \bigg[    \bigg| \frac{1}{N} \sum_{1 \le i < j \le N}  x_i x_j  \bigg|^2 \bigg] = \frac{\binom{N}{2}}{N^2} \le 1\,,
\]
and hence  the uniform integrability of the random variable sequence $\bigg(\bigg|\sum_{1 \le i < j \le N}  x_i x_j \bigg| \bigg)_{N \ge 1}$ (see
the remark done in \eqref{uni-int}).
Then, thanks to Theorem~\ref{bool-int} and to what we
explained in \eqref{thm: unif int + conv in dist implies conv in mean}, we see that
\[
\lim_{N \to \infty} \frac{\boldsymbol{\lambda}(\Bb_{=2}^N) }{N} = \lim_{N \to \infty} \frac{1}{N} \; \mathbb{E} \bigg( \bigg| \sum_{1 \le i < j \le N}  x_i x_j \bigg| \bigg) = \lim_{N \to \infty} \mathbb E \bigg( \frac{|Z_N^2 -1 |}{2} \bigg) =  \mathbb E \bigg( \frac{|Z^2 -1 |}{2}\bigg)\,.
\]
Computing the latter integral, we finally arrive at
\[
\lim_{N \to \infty} \frac{\boldsymbol{\lambda}(\Bb_{=2}^N) }{N} = \mathbb E \bigg( \frac{|Z^2 -1 |}{2} \bigg) = \frac{1}{\sqrt{2 \pi}} \int_{-\infty}^{\infty} \frac{|t^2-1|}{2}e^{\frac{-t^2}{2}} dt = \sqrt{\frac{2}{\pi e}}\,. \qedhere
\]
\end{proof}

The general case of arbitrary  degrees $d \in \NN$ is  technically more involved.
In the previous proof for $d=2$, the key step is to rewrite $$\frac{1}{N} \dis\sum_{1 \le i < j \le N}  x_i x_j $$
in terms of a  polynomial in one variable.

For  arbitrary $d$ (as in Theorem~\ref{thm: limit bool homog}), we require an adequate decomposition of the random variable
\begin{equation}\label{AvariableA}
  Y_N(x) = \frac{1}{N^{d/2}}\dis  \sum_{ |S| = d } x^S = \frac{1}{N^{d/2}}\dis  \sum_{ \alpha \in \Lambda_T(d,N)} x^\alpha\,.
\end{equation}

In order to derive the expectation of these kernels, we offer   two proofs with independently interesting features. Both
approaches have two steps (see also the proof of Theorem~\ref{thm: cte proy boole 2 hom}): in a  first step  the kernels
$Y_N$ are reformulated  in such a way that in a second step  the central limit theorem may be  applied properly. In both
approaches the  second steps  are  basically identical, whereas the arguments for the first ones are substantially different.
The first approach (see Section~\ref{firstapr1}) doesn't need any knowledge on Hermite polynomials. It is  mainly based on
a~natural decomposition of multi indices $\alpha$ into their 'even part' $\alpha_E$ and their 'tetrahedral part' $\alpha_T$.
At the very end we  arrive at the limit formula from \eqref{limform} discovering  'posthum' that it may be written in terms
of Hermite polynomials. Knowing this fact, one in a second approach may use a (somewhat classical) formula of Beckner from
\cite{beckner1975inequalities}  to reach the same goal. Since we use this formula  as  a sort of black box, this second approach here appears to be
shorter.

\subsection{First approach - decomposing indices} \label{firstapr1}

For any index $\alpha~\in~\Lambda(d,N)$ there are a unique integer $0 \leq k \leq d/2$ and a~unique decomposition
\[
\alpha = \alpha_T+ \alpha_E
\]
of $\alpha$ into the sum of a~tetrahedral index $\alpha_T \in \Lambda_T(d-2k,N)$ and an even index $\alpha_E \in \Lambda_E(2k,N)$
(so all coordinates of $\alpha_E$ are even and $|\alpha_E|=2k$).
This in particular implies that
\[
\text{
$x^\alpha = x^{\alpha_T} \cdot \underbrace{x^{\alpha_E}}_{=1} = x^{\alpha_T} $ \quad
for every \quad  $x \in\{-1, +1\}^N$.}
\]
The way to find such a decomposition of a  given index $\alpha \in \Lambda(d,N)$ is rather simple: Given $1 \leq j \leq N$,
the $j$-th coordinate of the tetrahedral part $\alpha_T \in \Lambda(d,N)$ is equal to 1, whenever $\alpha_j$ is odd, and 0 else.
The even part is defined as $\alpha_E:= \alpha - \alpha_T \in \Lambda(d,N)$. Defining
\[
k := \frac{|\alpha_E|}{2}\,,
\]
we indeed see that $\alpha_T \in \Lambda_T(d-2k,N)$ and $\alpha_E \in \Lambda_E(2k,N)$.
Moreover, since  all coordinates of $\alpha_E$ are even, there exists a unique $\beta \in \Lambda(d,N)$ such that $\alpha_E = 2 \beta$,
that is, there is a canonical way to identify $\Lambda_E(2k,N)$ and $\Lambda(k,N)$.

All together, this leads to the following identification of index sets:
\begin{equation} \label{trick}
\Lambda(d,N)	\leftrightsquigarrow \bigcup_{k=0}^{\lfloor d/2 \rfloor} \Lambda_T(d-2k,N) \times \Lambda_E(2k,N) \leftrightsquigarrow \bigcup_{k=0}^{\lfloor d/2 \rfloor} \Lambda_T(d-2k,N) \times \Lambda(k,N)\,,
\end{equation}
where the second identification comes from the fact that there is a canonical correspondence between  $\Lambda_E(2k,N)$ and $\Lambda(k,N)$.

We say that two indices $\alpha \in \Lambda(d_1,N)$ and $\beta \in \Lambda(d_2,N)$ do not share variables whenever  they have disjoint support.

\begin{lemma}\label{rem: separate variables}
Let  $\alpha \in \Lambda(d,N) $ and  $k \leq d/2$. Assume that the tetrahedral part $\alpha_T~\in~\Lambda_T(d-2k,N)$ and the  even part $\alpha_E \in \Lambda(2k,N)$ of $\alpha$ do not share variables, and that $ \alpha_E = 2 \beta $ with $\beta \in \Lambda_T(k,N)$. Then
\[
|[\alpha]| =  \frac{d!}{2^k}\,.
\]
\end{lemma}

\begin{proof}
  We deduce  from~\eqref{scholz} that
  \begin{equation*}
    |[\alpha]| = \frac{d!}{\alpha!} = \frac{d!}{(\alpha_T+ \alpha_E)!}=\frac{d!}{\alpha_T! \alpha_E!} =\frac{d!}{(2\beta)!} = \frac{d!}{2^k}. \qedhere
  \end{equation*}
\end{proof}

Let us begin analyzing the idea to rewrite the random variable from~\eqref{AvariableA} for arbitrary  degrees $d \ge 2$.
Taking $ \sum_{i = 1}^N x_i $ for $x \in\{-1, +1\}^N$ to the power $d$ and writing
$ x_\bi = x_{i_1}  \ldots x_{i_1} $ for  $\bi \in \mathcal{M}(d,N)$, we get
\[
\left( \sum_{i = 1}^N x_i \right)^d =\sum_{\bi \in \mathcal{M}(d,N)} x_{\bi}
=
\sum_{\bj \in \mathcal{J}(d,N)} \sum_{\bi \in [\bj]} x_\bi
=
\sum_{\bj \in \mathcal{J}(d,N)} |[\bj]|x_{\bj}
= \sum_{\alpha \in \Lambda(d,N)} |[\alpha]| x^{\alpha}\,.
\]
Decomposing each $\alpha \in \Lambda(d,N)$,
according to (the first identification in) \eqref{trick}, and  using the fact that $x_i^2 = 1$,
we have
\begin{equation*}
\left( \sum_{i = 1}^N x_i \right)^d = \sum_{k = 0}^{\floor{d/2}} \left(  \sum_{ \alpha_E \in \Lambda_E(2k,N)} \; \; \sum_{ \alpha_T \in \Lambda_T(d - 2k,N)} |[\alpha_T + \alpha_E]| x^{\alpha_T} \right)\,.
\end{equation*}
Consequently, by \eqref{scholz}
\begin{align*}
  \left( \sum_{i = 1}^N x_i \right)^d
  =
   \sum_{ \alpha_T \in \Lambda_T(d,N)} d! x^{\alpha_T}
  +
  \sum_{k = 1}^{\floor{d/2}} \left(  \sum_{ \alpha_E \in \Lambda_E(2k,N)} \; \; \sum_{ \alpha_T \in \Lambda_T(d - 2k,N)} |[\alpha_T + \alpha_E]| x^{\alpha_T} \right)\,,
\end{align*}
so that rearranging terms
leads to
\begin{equation}\label{eq: sum power d (2)}
\sum_{ \alpha_T \in \Lambda_T(d ,N)} x^{\alpha_T} = \frac{1}{d!} \left[ \left( \sum_{i = 1}^N x_i \right)^d - \sum_{k = 1}^{\floor{d/2}} \left(  \sum_{ \alpha_T \in \Lambda_T(d - 2k,N)}  \; \;
\sum_{ \alpha_E \in \Lambda_E(2k,N)} \; \;  |[\alpha_T + \alpha_E]|x^{\alpha_T} \right) \right].
\end{equation}

To illustrate this, note that for $d \in \{2,3\}$  we get
\begin{itemize}
\item $
\dis\sum_{1 \le i < j \le N}  x_i x_j = \frac{1}{2} \left[ \left( \sum_{i=1}^N x_i \right)^2 - N \right]$\,,
\item $
\dis\sum_{1 \le i < j < k \le N}  x_i x_j x_k = \frac{1}{6} \left[ \left( \sum_{i=1}^N x_i \right)^3 - (3N-2)\left( \sum_{i=1}^N x_i \right) \right]$\,.
\end{itemize}

The following two technical lemma analyze \eqref{eq: sum power d (2)} in more detail.

\begin{lemma}\label{lem: lim coef homog bool}
Given $d,N \in \NN$, we for each $x \in\{-1, +1\}^N$ have
\[
\sum_{ \alpha \in \Lambda_T(d ,N)} x^{\alpha} = \frac{1}{d!} \left[ \left( \sum_{i = 1}^N x_i \right)^d - \sum_{k = 1}^{\floor{d/2}} C_{d,k,N} \sum_{ \alpha_T \in \Lambda_T(d - 2k,N)} x^{\alpha_T}  \right]\,,
\]
where for $1 \le k \le \floor{d/2}$
\[
\text{$C_{d,k,N} = \binom{N-d+2k}{k} \frac{d!}{2^k} + D_{d,k,N} $ \,\,\,\,and\,\,\,\, $0 \leq D_{d,k,N} \le N^{k-1} 2 d d!$}\,.
\]
 In particular,
\begin{equation}\label{formu}
  \lim_{N \to \infty} \frac{C_{d,k,N}}{N^k} = \frac{d!}{k! \;2^k}.
\end{equation}
\end{lemma}

\begin{proof}
We fix some  $1 \le k \le \floor{d/2}$, and note that (in view of equation \eqref{eq: sum power d (2)}) we need to study
\[
\sum_{ \alpha_T \in \Lambda_T(d - 2k,N)}
 \; \;
\sum_{ \alpha_E \in \Lambda_E(2k,N)}|[\alpha_T + \alpha_E]| x^{\alpha_T}\,,
\]
in order to be able to first define and second control the factor $C_{d,k,N}$.
We fix some $\alpha_T \in \Lambda_T(d-2k,N)$, and start to decompose
\[
\sum_{ \alpha_E \in \Lambda_E(2k,N)}|[\alpha_T + \alpha_E]| x^{\alpha_T}\,.
\]
Let us denote the set of even indices which do not share variables with $\alpha_T$, by $\Lambda_E(\alpha_T) \subset \Lambda_E(2k,N)$,
and use $\Lambda_E(\alpha_T)^c \subset \Lambda_E(2k,N)$ for its complement in $\Lambda_E(2k,N)$. Then
\[
\sum_{ \alpha_E \in \Lambda_E(2k,N)}|[\alpha_T + \alpha_E]| x^{\alpha_T}
=
\bigg[
\underbrace{\sum_{ \alpha_E \in \Lambda_E(\alpha_T)^c}|[\alpha_T + \alpha_E]|}_{= :A} + \underbrace{\sum_{ \alpha_E \in \Lambda_E(\alpha_T)}|[\alpha_T + \alpha_E]| }_{= :B}\bigg]  x^{\alpha_T}\,,
\]
and we handle both sums separately.

Before we start to  estimate, note that $A + B$ does not depend on $\alpha_T$, that is, for each $\alpha_T$ the sum
\[
\sum_{ \alpha_E \in \Lambda_E(2k,N)}|[\alpha_T + \alpha_E]|
\]
is the same.
Indeed, given two different tetrahedral multi indices  $\alpha_T, \alpha_T' \in \Lambda_T(d-2k,N)$, the natural bijection between index sets that maps $\alpha_T$ to $\alpha_T'$, given by a suitable permutation of coordinates, also lets the sums invariant. This allows us to define $$C_{d,k,N} := A + B\,.
$$
\noindent Estimating $A:$ In order to estimate the cardinality of $\Lambda_E(\alpha_T)^c$,
observe  that any multi index in $\Lambda_E(\alpha_T)^c$ needs to share at least one of the $d-2k$ possible variables of $\alpha_T$, and therefore
\begin{equation} \label{bound-a}
  |\Lambda_E(\alpha_T)^c| \le (d-2k) |\Lambda_E(2k-2,N)| = (d-2k) |\Lambda(k-1,N)| \le (d-2k) N^{k-1}\,.
\end{equation}
Clearly, $\alpha_T + \alpha_E \in \Lambda(d,N)$ for any $\alpha_E \in \Lambda_E(\alpha_T)^c$, and hence
by~\eqref{scholz}
\[
 A= \sum_{ \alpha_E \in \Lambda_E(\alpha_T)^c}  |[\alpha_T + \alpha_E]|
= \sum_{ \alpha_E \in \Lambda_E(\alpha_T)^c} \frac{d!}{(\alpha_T + \alpha_E)!} \leq (d-2k) N^{k-1} d!\,.
\]

\noindent  Estimating $B:$
We have
\begin{align*}
B=\sum_{ \alpha_E \in \Lambda_E(\alpha_T)}  |[\alpha_T +\alpha_E]|
 =  \sum_{ \alpha_E \in \Lambda_E(2k,N-d+2k)}  |[\alpha_T +\alpha_E]|.
\end{align*}
We then may  decompose the index set $\Lambda_E(2k,N-d+2k)$ into the set of those indices which use $k$ variables, denoted by $\Lambda_E(2k,N-d+2k)^k$, and the set that contains all even indices with less than $k$ variables, denoted by $\Lambda_E(2k,N-d+2k)^{<k}$,
so
\begin{align*}
B
 = \underbrace{\sum_{ \alpha_E \in \Lambda_E(2k,N-d+2k)^{< k}}  |[\alpha_T +\alpha_E]|}_{= :B^{< k}} \,+\, \underbrace{\sum_{ \alpha_E \in \Lambda_E(2k,N-d+2k)^{k}}  |[\alpha_T +\alpha_E]|}_{= :B^{=k}}\, 
 .
\end{align*}
 Observe that given a multi index in $\Lambda_E(2k,N-d+2k)^{<k}$, it is mandatory that some variable appears to at least the $4$th power (since all the indices in the set $\Lambda_E(2k,N-d+2k)^{<k}$ are even). Going through all the possible $N-d +2k$ variables, we get the bound
\begin{align*}
|\Lambda_E(2k,N-d+2k)^{<k}| & \le (N-d+2k) |\Lambda_E(2k-4,N-d+2k)| \\
& = (N-d+2k) |\Lambda(k-2,N-d+2k)| \le (N-d+2k)^{k-1}\,,
\end{align*}
and  then (as above with \eqref{scholz})
\[
B^{< k} = \sum_{ \alpha_E \in \Lambda_E(2k,N-d+2k)^{<k}} |[\alpha_T + \alpha_E]| \le \left| \Lambda_E(2k,N-d+2k)^{<k} \right| \; d! \le N^{k-1} d!\,.
\]
On the other hand, note that for each $ \alpha \in \Lambda_E(2k,N-d+2k)^{k}$  there is $\beta \in \Lambda_T(k,N-d+2k)$ such that $\alpha = 2 \beta$. Thus, this defines a~one to one mapping between $\Lambda_E(2k,N-d+2k)^{k}$ and $\Lambda_T(k,N-d+2k)$, so we get
$$|\Lambda_E(2k,N-d+2k)^{k}| = |\Lambda_T(k,N-d+2k)| = \binom{N-d+2k}{k}.$$
 Then by   Lemma~\ref{rem: separate variables}
\begin{align*}
B^{= k} = \sum_{ \alpha_E \in \Lambda_E(2k,N-d+2k)^{k}}  |[\alpha_T+\alpha_E] |
=  \binom{N-d+2k}{k} \frac{d!}{2^k}\,.
\end{align*}

\noindent  Combining step:
We define
$D := A + B^{< k} $
(note that $D = D_{d,k,N}$ in fact depends on $d,k$ and $N$).
Then $ D\leq  N^{k-1}  2d d!$, and  all in all we obtain

\begin{align*}
 C_{d,k,N} = \sum_{ \alpha_E \in \Lambda_E(2k,N)}   |[\alpha_T + \alpha_E]|
    =
    A + B
       =
    B^{= k} + D_{d,k,N}\leq   \binom{N-d+2k}{k}\frac{d!}{2^k}+ N^{k-1}  2d d!\,.
\end{align*}
Finally, for a fixed $k$, a standard calculation gives~\eqref{formu}.
\end{proof}

The following lemma goes one  step further -   namely, rewriting the  random variable from~\eqref{AvariableA}
in a~way that later allows us to calculate the limit of its mean.
Notice that for $d = 0$ this random variable equals the constant function of value $1$.

\begin{lemma}\label{lem: boolean polynomial rewriting}
Given $d \in \NN_0$ and $N \in \NN$, there is  $\varphi_{d,N} \in C(\RR)$ such that for all
$x \in\{-1, +1\}^N$
\begin{equation}\label{equ}
  \frac{1}{N^{d/2}}\dis  \sum_{ \alpha \in \Lambda_T(d,N)} x^\alpha = P_d \left( \frac{\sum_{i=1}^N x_i}{\sqrt{N}} \right) + \varphi_{d,N} \left(\frac{\sum_{i=1}^N x_i}{\sqrt{N}} \right)\,,
\end{equation}
where  $P_d$ is as in Theorem \ref{thm: limit bool homog}
and
$\varphi_{0,N} = \varphi_{1,N} =0$.
Moreover, we have that
\begin{equation} \label{limitA}
    \varphi_{d,N} \left(\frac{\sum_{i=1}^N x_i}{\sqrt{N}} \right)
    {\overset{P}{\longrightarrow}}  0 \,,  \quad\, \text{as} \,\, N\to \infty\,,
\end{equation}
and
\begin{equation} \label{limit2}
    \frac{1}{N^{d/2}}\dis  \sum_{ \alpha \in \Lambda_T(d,N)} x^\alpha  
    {\overset{D}{\longrightarrow}} P_d(Z) \,,  \quad\, \text{as} \,\, N\to \infty\,,
\end{equation}
where $Z$ is a normal distribution with mean $0$ and variance $1$.
\end{lemma}

\begin{proof}

The proof will be by induction on $d$. Recall from Theorem~\ref{thm: limit bool homog} that $P_0 =1$ and $P_1=t$.
Then for  $d=0$ there is nothing to prove, and for $d=1$ the proof by~\eqref{CLT} is obvious as well.

Let us fix some $d \ge 2$, and  assume that the result is true for all degrees $\leq d-1$.
The aim is to prove the result for  $d$. Dividing the equality from  Lemma~\ref{lem: lim coef homog bool}
by $N^{d/2} = N^{k}N^{(d-2k)/2}$ and using the  inductive hypothesis, we have
\begin{align*}
\frac{1}{N^{d/2}} \sum_{ \alpha \in \Lambda_T(d ,N)} x^{\alpha} & = \frac{1}{d!} \left( \frac{\sum_{i = 1}^N x_i}{\sqrt{N}} \right)^d - \sum_{k = 1}^{\floor{d/2}}  \frac{C_{d,k,N} }{d! N^k} \left(\frac{1}{N^{(d-2k)/2}} \sum_{ \alpha_T \in \Lambda_T(d - 2k,N)} x^{\alpha_T} \right)\\[1ex]&
=  \frac{1}{d!} \left( \frac{\sum_{i = 1}^N x_i}{\sqrt{N}} \right)^d  - \sum_{k = 1}^{\floor{d/2}}  \frac{C_{d,k,N} }{d! N^k} \left(P_{d-2k}\left( \frac{\sum_{i = 1}^N x_i}{\sqrt{N}} \right) + \varphi_{d-2k,N}\left( \frac{\sum_{i = 1}^N x_i}{\sqrt{N}} \right) \right).
\end{align*}
Defining for $t \in \mathbb{R}$
\[
\varphi_{d,N}(t) := \sum_{k = 1}^{\floor{d/2}}  \left(\frac{C_{d,k,n} }{d! N^k} - \frac{1}{k! 2^k} \right)\left(P_{d-2k}(t) + \varphi_{d-2k,N}\left( t \right) \right)
\]
and recalling the definition of $P_d$,
we see  that \eqref{equ}  holds.
It remains to show the two limit
formulas from~\eqref{limitA} and~\eqref{limit2}.
By the inductive hypothesis for each $1 \leq  k \leq d/2$
\begin{equation*}
 P_{d-2k}\left(\frac{\sum_{i=1}^N x_i}{\sqrt{N}} \right)
{\overset{D}{\longrightarrow}} P_{d-2k}(Z)
\quad
\text{and}
\quad
    \varphi_{d-2k,N} \left(\frac{\sum_{i=1}^N x_i}{\sqrt{N}} \right)
        {\overset{P}{\longrightarrow}}  0 \,,  \quad\, \text{as \,\, $N\to \infty$\,,}
\end{equation*}
Moreover, by
 Lemma \ref{lem: lim coef homog bool}, \eqref{formu} we have that
  $
 \dis\lim_{N \to \infty} \frac{C_{d,k,n} }{d! N^k} - \frac{1}{k! 2^k} = 0\,,
 $
so that \eqref{limitA} follows by Slutsky's theorem.
 Since convergence in distribution
 is preserved under continuous
 functions,
we conclude that
\[ P_{d}\left(\frac{\sum_{i=1}^N x_i}{\sqrt{N}} \right)
{\overset{D}{\longrightarrow}} P_{d}(Z)\,, \quad\, \text{as \,\, $N\to \infty$\,,}
\]
and then we obtain~\eqref{limit2} from another application of Slutsky's theorem, using~\eqref{equ} and~\eqref{limitA}.
\end{proof}

Finally, we come to the proof of the main contribution of this section, Theorem~\ref{thm: limit bool homog},
which extends  Theorem~\ref{thm: cte proy boole 2 hom} to  all possible degrees.

\begin{proof}[Proof of  Theorem~\ref{thm: limit bool homog}]

We for $N \in \mathbb{N}$ define the random variable
$ Y_N(x)$ as in~\eqref{AvariableA}\,.
Applying Lemma~\ref{lem: boolean polynomial rewriting}, the central limit theorem as in~\eqref{CLT}, the fact that convergence in distribution is preserved under  continuous functions, and Slutsky's theorem, we get
\[
Y_N(x)  = P_d \left( \frac{\sum_{i=1}^N x_i}{\sqrt{N}} \right) + \varphi_{d,N} \left(\frac{\sum_{i=1}^N x_i}{\sqrt{N}} \right)
{\overset{D}{\longrightarrow}} P_d(Z) \,, \quad\, \text{as \,\, $N\to \infty$}\,.
\]

Now orthogonality of the Fourier--Walsh basis assures that for all $N$
\[
\mathbb{E}|Y_N|^2 = \frac{|\Lambda_T(d,N)|}{N^d} \le 1\,,
\]
which gives the uniform integrability of all $ Y_N $ (see the remark from
\eqref{uni-int}). Using  \cite[Theorem 3.5]{billingsley2013convergence}, this implies
\[
\lim_{N \to \infty}   \mathbb E |Y_N| = \mathbb E [|P_d(Z)|] = \frac{1}{\sqrt{2 \pi}} \int_{-\infty}^{\infty} |P_d(t)|e^{-\frac{t^2}{2}} dt\,,
\]
which by  Theorem~\ref{bool-int}  finishes the proof of \eqref{limform}.
The second claim follows with similar arguments. Observe first that by another application of Theorem~\ref{bool-int} we have
\[
\boldsymbol{\lambda}(\mathcal{B}_{\le d}^N)  = \mathbb E \left[ \left| \dis\sum_{  \alpha \in \Lambda_{T}(\le d,N) } x^\alpha \right| \right].
\]
Also by Lemma \ref{lem: boolean polynomial rewriting},
\eqref{limit2}, for every $0 \le k \le d$
\[
\frac{1}{N^{k/2}} \dis\sum_{\alpha \in \Lambda_{T}(k,N) } x^\alpha
{\overset{D}{\longrightarrow}} P_k(Z) \,, \quad\, \text{as \,\, $N\to \infty$\,,}
\]
where $Z$ is as above. Hence we as before  use  Slutsky's theorem and the fact that
 a sequence of random variables converges  in probability
whenever it converges in distribution to a constant, to see that for every $0 \le k < d $
\[
\frac{1}{N^{d/2}} \dis\sum_{  \alpha \in \Lambda_{T}(k,N) } x^\alpha
{\overset{P}{\longrightarrow}} 0 \,, \quad\, \text{as \,\, $N\to \infty$\,.}
\]
Now one more  application of  Slutsky's theorem shows that
\[
\lim_{N \to \infty} \frac{1}{N^{d/2}} \dis\sum_{ \alpha \in \Lambda_{T}(\le d,N) } x^\alpha
=
\lim_{N \to \infty} \frac{1}{N^{d/2}} \dis\sum_{k=1 }^d\dis\sum_{  \alpha \in \Lambda_{T}(k,N) } x^\alpha
=\lim_{N \to \infty} \frac{1}{N^{d/2}} \dis\sum_{  \alpha \in \Lambda_{T}( d,N) } x^\alpha
= P_d(Z),
\]
where the limit is taken in the sense of  distribution. Again by orthogonality
\[
 \mathbb{E} \left[    \left| \frac{1}{N^{d/2}} \sum_{ \alpha \in \Lambda_T(\leq d,N)} x^\alpha \right|^2 \right] = \frac{|\Lambda_T(\leq d,N)|}{N^d} < \infty \,,
 \]
implying that the random variables  $\frac{1}{N^{d/2}} \dis\sum_{  \alpha \in \Lambda_{T}(\leq d,N) } x^\alpha,\,\, N \in \mathbb{N}$
are uniformly integrable, and this is enough to conclude  that
\[
\dis\lim_{N \to \infty} \frac{\boldsymbol{\lambda}(\mathcal{B}_{\le d}^N)}{N^{d/2}} = \EE \left(|P_d(Z)| \right)
\]
(see again \eqref{thm: unif int + conv in dist implies conv in mean}). Together with the first claim this finishes the proof.
\end{proof}

\subsection{Second approach - Beckner's formula}  \label{firstapr2}
The following identity of  Beckner from \cite[Equation (5)]{beckner1975inequalities}
 re\-phrases $\sum_{ |S| = d} x_S$
directly in terms of Hermite polynomials, and may hence serve as a substitute of
Lemma~\ref{lem: lim coef homog bool}.

\begin{lemma} \label{thm: Beckner}
For each $d,N \in \NN$ and $x \in  \{-1,+1\}^N$
\[
\frac{1}{N^\frac{d}{2}}\sum_{ |S| = d} x_S = \frac{1}{d!} \left( h_d\left( \frac{x_1 + \ldots + x_N}{\sqrt{N}} \right) + \frac{1}{N} \sum_{k = 1}^{\floor{d/2}} a_{d,k,N} h_{d-2k}\left(\frac{x_1 + \ldots + x_N}{\sqrt{N}}\right) \right),
\]
where the coefficients $a_{d,k,N}$ are bounded in  $N$ and $h_\ell$ is the $\ell$-th Hermite polynomial.
\end{lemma}

With this the proof of  Theorem~\ref{thm: limit bool homog} follows
similarly  as above replacing the polynomials $P_d$ by $h_d/d!$.
Indeed, applying Lemma~\ref{thm: Beckner}, the fact that convergence in distribution is preserved under  continuous functions
together with~\eqref{CLT} gives that for all $\ell \in \mathbb{N}_0$ we 
again  get
\[
h_{\ell} \left( \frac{\sum_{i=1}^N x_i}{\sqrt{N}} \right) \overset{D}{\longrightarrow} h_{\ell}(Z) \,, \quad\, \text{as \,\, $N\to \infty$}\,.
\]
Also, as the constants $a_{d,k,N}$ are bounded in $N$, we have $\lim_{ N \to \infty} \frac{a_{d,k,N}}{N} = 0$ for every $1 \le k \le \floor{d/2} $, so that by Slutky's theorem ($Y_N$ again as in~\eqref{AvariableA})
\[
Y_N(x)  =  \frac{1}{d!} \left( h_d\left( \frac{x_1 + \ldots + x_N}{\sqrt{N}} \right) + \frac{1}{N} \sum_{k = 1}^{\floor{d/2}} a_{d,k} h_{d-2k}\left(\frac{x_1 + \ldots + x_N}{\sqrt{N}}\right) \right)
\,\,{\overset{D}{\longrightarrow}} \,\,\frac{1}{d!}h_d(Z) \,, \quad\, \text{as \,\, $N\to \infty$.}
\]
 The rest of the proof
proceeds in a~similar way as before.

While it may seem that this approach is more concise, it actually conceals the use of a deep specific result as if it were a 'hidden toolkit'. Furthermore, we would like to underscore that the first approach, founded on the contemporary combinatorial technique known as 'monomial decomposition' (see e.g., 
\cite{galicer2021monomial,mansilla2019thesis}), reinforces this fresh perspective. We firmly believe that this approach holds substantial potential for wider applications.

\section{Sidon  vs projection constant}
\label{Unconditional basis vs projection constant}

Given a set  $\mathcal{S} \subset \big\{ S \subset [N]\big\}$, we now  establish an intimate link of the Sidon constant of $\mathcal{B}^N_{\mathcal{S}}$ with its projection constant. To do so, we define
\[
\mathcal{S}^\flat = \big\{S \setminus \{i\} \colon S \in \mathcal{S}, \, i \in S  \big\}\,.
\]
Moreover, the  following constant
\begin{equation}\label{kappa}
    \kappa:=\bigg(\prod_{k=1}^\infty\
\sinc{\frac{\pi}{\mathfrak{p}_k}}\bigg)^{-1}=2.209\ldots\,,
\end{equation}
is going to appear, where  $(\mathfrak{p}_k)_{k\geq 1}$ is
the  sequence of prime numbers and $\sinc x:=(\sin x)/x$.

The main result is as follows.

\begin{theorem} \label{BGL3}
 Let $2 \leq d \leq N$ . Then
  \begin{equation}\label{est-one}
    \boldsymbol{\sid}\big(\mathcal{B}^N_{=d}\big) \,\,\leq\,\,
     C(d)  \,
          \boldsymbol{\lambda}\big(\mathcal{B}^N_{=d-1}\big)\,,
  \end{equation}
    where $C(d)\leq e^{d} (2d) \kappa^d 2^{d-1}$. Additionally,
    \begin{equation}\label{est-two}
 \boldsymbol{\sid}\big(\mathcal{B}^N_{\leq d}\big) \,\,\leq\,\,
    \widetilde{C}(d) \, \max_{1 \leq  k \leq d}  \boldsymbol{\lambda}\big(\mathcal{B}^N_{= k-1}\big)\,,
      \end{equation}
  where $\widetilde{C}(d)\leq  (2.69076)^{2d} (d+1) (1+\sqrt{2})^d (2d) \kappa^d 2^{d}$.
  \end{theorem}

Note that the estimates of $\boldsymbol{\sid}\big(\mathcal{B}^N_{= d}\big)$ and $\boldsymbol{\sid}\big(\mathcal{B}^N_{\leq d}\big)$ are anyway trivial if $0\leq d \leq 1$, so there is no need to compare them with projection constants.
In fact, the previous the statement is an obvious consequence of the following more general result.

\begin{theorem} \label{BGL2}
Let $2 \leq d \leq N$ and $\mathcal{S} \subset \big\{ S \subset [N] \colon |S| = d \big\}$. Then
\[
\boldsymbol{\sid}\big(\mathcal{B}^N_{\mathcal{S}}\big) \,\,\leq\,\,
C(d) \,\,\|\mathbf{Q}: \mathcal{B}^N_{=d} \to \mathcal{B}^N_{\mathcal{S}}\|
\,\, \boldsymbol{\lambda}\big(\mathcal{B}^N_{\mathcal{S}^\flat}\big)\,,
\]
where  the constant  $C(d)$ is   as in~\eqref{est-one} and $\mathbf{Q}$ is the projection annihilating Fourier coefficients with indices $S$ not in~$\mathcal{S}$.

More generally, if $\mathcal{S} \subset \big\{ S \subset [N] \colon |S| \leq d \big\}$, then
\[
\boldsymbol{\sid}\big(\mathcal{B}^N_{\mathcal{S}}\big) \,\,\leq\,\,
\widetilde{C}(d) \,\,\max_{1 \leq  k \leq d}\|\mathbf{Q}_k: \mathcal{B}^N_{=k} \to \mathcal{B}^N_{\mathcal{S}_{=k}}\|
\max_{1 \leq  k \leq d}  \boldsymbol{\lambda}(\mathcal{B}^N_{(\mathcal{S}_{=k})^\flat})\,,
\]
where $\mathcal{S}_{=k}=\{S \in \mathcal{S} :\, |S|=k \}$, the constant  $\widetilde{C}(d)$ is as in~\eqref{est-two} and $\mathbf{Q}_k$ is the projection annihilating Fourier coefficients with indices $S$ not in~$\mathcal{S}_{=k}$.
\end{theorem}

The proof of Theorem~\ref{BGL2} is given in Section~\ref{Gordon--Lewis vs projection constant}. In the coming two sections
we prepare it collecting a few  independently interesting facts.

\subsection{Annihilating coefficients}
\label{Annihilating coefficients}

 In what follows we need a lemma, which is a discrete variant of a~result
 for polynomials on the $N$-dimensional torus due  to  Ortega-Cerd\`a, Ouna\"{\i}es and Seip in \cite{ortega2009sidon}. We include a proof for the sake of completeness.

\begin{lemma} \label{OrOuSe}
Let $2 \leq d \leq N$ .
Then
\[
\big\|\mathbf{Q}\colon \mathcal{P}_d (\ell^N_\infty) \to \mathcal{B}^N_{=d}\big\|
\leq \kappa^d 2^{d-1}  \,,
\]
where $\mathbf{Q}$ is the projection annihilating coefficients with non-tetrahedral indices. In particular,
\[
\boldsymbol{\lambda}\big(\mathcal{B}^N_{=d}\big) \leq \kappa^d 2^{d-1} \boldsymbol{\lambda}\big(\mathcal{P}_d (\ell^N_\infty)\big)\,.
\]
\end{lemma}

\begin{proof}
As usual, we write $\pi(x)$ for the counting function of the prime numbers.  Now, given
$$t=(t_1,\ldots,t_{\pi(d)})  \in  \mathcal{R} :=[0,1]^{\pi(d)}\,,$$
  define
\[
r_d(t)=c_d \exp \left(2\pi i\Bigl(\frac {t_1}2 + \frac {t_2}3+\cdots+ \frac
{t_{\pi(d)}}{\mathfrak{p}_{\pi(d)}}\Bigr)\right),
\]
where
\[
c_d=\prod_{k=1}^{\pi(d)} \left(\frac {\mathfrak{p}_k}{2\pi i}
\Bigl(e^{\frac {2\pi i}{\mathfrak p_k}} -1\Bigr)\right)^{-1}.
\]
Note that the function $r_d\colon \mathcal{R}\to \mathbb{C}$ has the following properties:
\begin{enumerate}
\item[(i)] $\int_{\mathcal{R}} r_d(t)\,d\mu(t)=1$,
\item[(ii)] $\int_{\mathcal{R}} r_d^k(t)\, d\mu(t)=0$\ \ for each \, $2 \leq k \leq d$,
\item[(iii)] $|r_d(t)|\le \kappa$\ \ for all $t \in \mathcal{R}$;
\end{enumerate}
here $d\mu$ denotes the  Lebesgue measure on $\mathcal{R}$.
Indeed,  (i) and (ii)
are trivial and follow by the definition of the function, and (iii) holds because
$|r_d(t)| = |c_d|$ and
\[
|c_d|^{-2}= \prod_{k=1}^{\pi(d)} \frac{\mathfrak{p}_k^2}{(2\pi)^2}
\Bigl|e^{\frac{2\pi i}{\mathfrak{p}_k}}-1\Bigr|^2
 = \prod_{k=1}^{\pi(d)} \sinc^2\frac{\pi}{\mathfrak{p}_k}.
\]
Take now some  $ P \in \mathcal{P}_d (\ell^N_\infty)$ and a representation
$Px = \sum_{|\alpha| = d} c_\alpha x^\alpha, \, x \in \mathbb{R}^N$. Then  by the properties (i) and (ii) we have
\[
\mathbf{Q}P(x)=\int_{{\mathcal{R}}^N} P_{\mathbb{C}}(x_1r_d(t^1),\ldots, x_n r_d(t^N))\,
d\mu(t^1)\cdots d\mu(t^n)\,,\,\,\, x \in \ell^N_\infty.
\]
where $P_\mathbb{C}(z) = \sum_{|\alpha| = d} c_\alpha z^\alpha, \, z \in \mathbb{C}^N$.
By (iii) we deduce that
$$|P_\mathbb{C}(x_1 r_d(t^1),\ldots, x_N
r_d(t^N))| \le \kappa^d \|P_\mathbb{C}\|_{\mathcal{P}_d (\ell^N_\infty(\mathbb{C}))}$$
for every $x~\in~B_{\ell^N_\infty}$, and therefore
\[
\Vert \mathbf{Q}(P) \Vert_{\mathcal{B}^N_{=d}} \leq \kappa^d \,\Vert P_\mathbb{C} \Vert_{\mathcal{P}_d (\ell^N_\infty(\mathbb{C}))}.
\]
But by a result of Visser \cite{visser1946generalization} we know that
\[
\Vert P_\mathbb{C} \Vert_{\mathcal{P}_d (\ell^N_\infty(\mathbb{C}))} \leq
2^{d-1} \Vert P \Vert_{\mathcal{P}_d (\ell^N_\infty)}\,.
\]
All together this  proves the first statement; the second one is then an immediate consequence.
\end{proof}

\subsection{Sidon  vs Gordon--Lewis constant} \label{GL}
The following fact is the first of two major steps towards the proof of Theorem~\ref{BGL2}.

\begin{theorem} \label{BGL1}
Let $2 \leq d \leq N$ and $\mathcal{S} \subset \big\{ S \subset [N] \colon |S| = d \big\}$. Then
 \begin{equation}
\boldsymbol{g\!l}\big(\mathcal{B}^N_{\mathcal{S}}\big)\,\,
  \leq\,\,
   \boldsymbol{\chi}\big(\mathcal{B}^N_{\mathcal{S}}\big)
    \,\,\leq\,\,
   \boldsymbol{\sid}\big(\mathcal{B}^N_{\mathcal{S}}\big) \,\,\leq\,\,
  e^{d}\boldsymbol{g\!l}\big(\mathcal{B}^N_{\mathcal{S}}\big)\,.
\end{equation}
 Additionally, if $\mathcal{S} \subset \big\{ S \subset [N] \colon |S| \leq d \big\}$ then 
  \begin{equation}
 \boldsymbol{g\!l}\big(\mathcal{B}^N_{\mathcal{S}}\big)\,\,
  \leq\,\,
   \boldsymbol{\chi}\big(\mathcal{B}^N_{\mathcal{S}}\big)
    \,\,\leq\,\,
   \boldsymbol{\sid}\big(\mathcal{B}^N_{\mathcal{S}}\big) \,\,\leq\,\,
  (2.69076)^{2d}\boldsymbol{g\!l}\big(\mathcal{B}^N_{\mathcal{S}}\big)\,.
  \end{equation}

\end{theorem}

Observe  that the first  estimate in Theorem~\ref{BGL1} is immediate  from the Gordon--Lewis inequality as formulated in  \eqref{gl-inequality}, and the second one is trivial.

The proof of the third estimate is  more involved and needs preparation. The first  lemma needed  is taken from \cite[Proposition 21.14]{defant2019libro} (its roots have to  be traced back to  the works \cite{pisier1978some} and \cite{schutt1978projection}).

\begin{lemma} \label{tool2}
Let $X_n$  be an $n$-dimensional Banach space with a~basis  $(x_k)_{k=1}^n$, and  denote the coefficient functionals of this  basis by $(x_k^\ast)$. Suppose that there exist constants $K_1,K_2\ge 1 $ such that for every choice of $\lambda, \mu \in \mathbb{C}^n$ the two diagonal
operators
\begin{align*}
&
 D_\lambda\colon  X_n \rightarrow \ell_2^n\,, \quad \, x_k \mapsto \lambda_k e_k
  \\
  &
 D_\mu^{*} \colon  X_n^{*} \rightarrow \ell_2^n \,, \quad \,x^{*}_k \mapsto \mu_k e_k
 \end{align*}
satisfy
$
\pi_1(D_\lambda)  \le K_1 \Big\| \sum_{k=1}^n \lambda_k x_k^{*}  \Big\|_{X_n^{*}}
$
and
$
\pi_1(D_\mu^{*})  \le K_2 \Big\| \sum_{k=1}^n \mu_k x_k  \Big\|_{X_n}\,.
$
Then
\[
 \boldsymbol{\chi}\big((x_k),X_n\big) \le K_1  K_2  \,\boldsymbol{g\!l}(X_n)\,.
\]
\end{lemma}

The second lemma is an almost immediate consequence of the so-called hypercontractivity of functions on the Boolean cube.

\begin{lemma} \label{abssum}
Let $2 \leq d \leq N$ and $\mathcal{S} \subset \big\{ S \subset [N] \colon |S| = d \big\}$. Then 
  \begin{equation}\label{onesummingB}
   \pi_1 \big(I \colon  \mathcal{B}^N_\mathcal{S} \to  \ell_2(\mathcal{S}) \big)  \leq e^{\frac{d}{2}}\,,
\end{equation}
where $I(f) = (\hat{f}(S))_{S \in \mathcal{S}}$ for $f \in \mathcal{B}^N_\mathcal{S} $. More generally, if $\mathcal{S} \subset \big\{ S \subset [N] \colon |S| \leq d \big\}$, then 
  \begin{equation}\label{onesummingA}
   \pi_1 \big(I \colon  \mathcal{B}^N_\mathcal{S} \to  \ell_2(\mathcal{S}) \big)  \leq (2.69076)^d\,.
\end{equation}

\end{lemma}

\begin{proof}
Suppose that $\mathcal{S} \subset \big\{ S \subset [N] \colon |S| = d \big\}$ and take finitely many $f_1, \dots, f_M \in \mathcal{B}^N_\mathcal{S} $. 
Then by~\eqref{homogeneousweissler} we have
\begin{align*}
 \sum_{k=1}^{M}  \|I f_k\|_{\ell_2(\mathcal{S})}
 &
 =
  \sum_{k=1}^{M}  \|f_k\|_{L_2(\{-1,+1\}^N)}
  \\&
  \leq
  e^{\frac{d}{2}} \sum_{k=1}^{M}  \|f_k\|_{L_1(\{-1,+1\}^N)} = e^{\frac{d}{2}} \sum_{k=1}^{M} \mathbb{E} [|f_k|]  \\&
  =  e^{\frac{d}{2}} \mathbb{E} [\sum_{k=1}^{M}|f_k|] \leq
  e^{\frac{d}{2}}  \sup_{x \in \{-1,+1\}^N} \sum_{k=1}^{M}  |f_k(x)|
       \leq
  e^{\frac{d}{2}}
    \sup_{\varphi \in B_{(\mathcal{B}^N_\mathcal{S})^\ast}}  \sum_{k=1}^{M}  |\varphi(f_k)|\,,
\end{align*}
which gives~\eqref{onesummingB}. With the same argument and the use of~\eqref{weissler} instead of~\eqref{homogeneousweissler} we obtain~\eqref{onesummingA}.
\end{proof}

\begin{proof}[Proof of Theorem~\ref{BGL1}]
As explained above, we may concentrate on the third  estimate. For  $A \in \mathcal{S}$ we write
$
\psi_A \colon  \mathcal{B}^N_\mathcal{S} \rightarrow \mathbb{C}
$
for the coefficient functionals of the canonical basis $(\chi_A)_{A \in \mathcal{S}}$ in $\mathcal{B}^N_\mathcal{S}$.  They
 form the orthogonal  basis of the  dual $(\mathcal{B}^N_\mathcal{S})^\ast$ in  the sense that
\[
\langle \chi_A, \psi_B \rangle_{\mathcal{B}^N_\mathcal{S},(\mathcal{B}^N_\mathcal{S})^\ast} = \delta_{A,B}\,.
\]
Given  two real sequences  $\lambda= (\lambda_A)_{A \in \mathcal{S}}$ and $(\mu_{A})_{A \in \mathcal{S}}$,  we consider the two  diagonal operators
\begin{align*}
&
D_\lambda:   \mathcal{B}^N_\mathcal{S} \longrightarrow  \mathcal{B}^N_\mathcal{S}\,,
\quad \, D_\lambda(\chi_A) =  \lambda_A \chi_A
\\&
D_\mu^{*} \colon (\mathcal{B}^N_\mathcal{S})^\ast   \longrightarrow \mathcal{B}^N_\mathcal{S} \,,
\quad \, D_\mu^{*}(\psi_A) =  \mu_A \chi_A\,,
\end{align*}
and show that
\begin{gather}
\|D_\lambda \| \leq \Big\| \sum_{A \in \mathcal{S}} \lambda_A \psi_A \Big\|_{(\mathcal{B}^N_\mathcal{S})^\ast} \label{tonelli} \\
\|D_\mu^{*} \| \leq \Big\| \sum_{A \in \mathcal{S}} \mu_A \chi_A \Big\|_{\mathcal{B}^N_\mathcal{S} }  \,. \label{hobson}
\end{gather}
If we combine these two estimates with  Lemma~\ref{abssum}, together with the ideal property of the $\pi_1$-norm  and also Lemma~\ref{tool2}, then the conclusion follows.
Let us prove~\eqref{tonelli}. Define for
$x \in \{-1, +1\}^N $ the diagonal map
\[
D_x\colon \{-1,+1\}^N \rightarrow \{-1, +1\}^N,\,\,\, y \mapsto xy\,,
\]
and note that
$
\chi_A\circ D_x = x^A \chi_A
$
for every $A \in \mathcal{S}$.
Then for  $x \in \{-1,+1\}^N$ and $f=\sum_{A \in
\mathcal{S}}  \widehat{f}(A)  \chi_A \in \mathcal{B}^N_\mathcal{S}$ we get
\begin{align*}
\Big| \Big[D_\lambda \big( \sum_{\mathcal{S}}  \widehat{f}(A)  \chi_A \big)\Big] (x) \Big|
& = \Big| \sum_{\mathcal{S}} \lambda_A  \widehat{f}(A)  \chi_A(x)\Big| = \Big| \sum_{\mathcal{S}} \lambda_A  \widehat{f}(A)  x^A\Big| \\
& = \Big|  \langle \sum_{\mathcal{S}}  \widehat{f}(A) x^A \chi_A,  \sum_{\mathcal{S}} \lambda_A \psi_A\rangle _{\mathcal{B}^N_\mathcal{S},(\mathcal{B}^N_\mathcal{S})^\ast} \Big|\\
& = \Big|  \langle  \sum_\mathcal{S}  \widehat{f}(A) \chi_A \circ D_x ,  \sum_{\mathcal{S}} \lambda_A \psi_A\rangle _{\mathcal{B}^N_\mathcal{S},(\mathcal{B}^N_\mathcal{S})^\ast} \Big|
\\&
\leq
\Big\| \sum_{\mathcal{S}} \widehat{f}(A) \chi_A \circ D_x     \Big\|_{\mathcal{B}^N_\mathcal{S}}  \Big\| \sum_{\mathcal{S}} \lambda_A \psi_A   \Big\|_{(\mathcal{B}^N_\mathcal{S})^{*}}
\\&
=
\Big\| \sum_{\mathcal{S}} \widehat{f}(A) \chi_A    \Big\|_{\mathcal{B}^N_\mathcal{S}}  \Big\| \sum_{\mathcal{S}} \lambda_A \psi_A   \Big\|_{(\mathcal{B}^N_\mathcal{S})^{*}}\,.
\end{align*}
Obviously, this leads to the estimate from~\eqref{tonelli}, and to see~\eqref{hobson} we simply repeat the argument.
This completes the proof.
\end{proof}

\subsection{Gordon--Lewis vs projection constant} \label{Gordon--Lewis vs projection constant}
With the following theorem we establish the  second  step for the proof of Theorem~\ref{BGL2}.
If one combines it with Theorem~\ref{BGL1}, the proof of Theorem~\ref{BGL2} is immediate.

\begin{theorem} \label{gl_versus_proj}
Let $2 \leq d \leq N$ and $\mathcal{S} \subset \big\{ S \subset [N] \colon |S| = d \big\}$. Then
\[
\boldsymbol{g\!l}\big(\mathcal{B}^N_{\mathcal{S}}\big) \,\,\leq\,\,
c(d) \,\,\big\|\mathbf{Q}\colon \mathcal{B}^N_{=d} \to \mathcal{B}^N_{\mathcal{S}}\big\|
\,\, \boldsymbol{\lambda}\big(\mathcal{B}^N_{\mathcal{S}^\flat}\big)\,,
\]
where  $c(d) \leq (2d) \,  \kappa^d 2^{d-1}$ and $\mathbf{Q}$ denotes the projection annihilating Fourier coefficients with indices $S$ not in~$\mathcal{S}$.
More generally, if $\mathcal{S} \subset \big\{ S \subset [N] \colon |S| \leq d \big\}$, then
\[
\boldsymbol{g\!l}\big(\mathcal{B}^N_{\mathcal{S}}\big) \,\,\leq\,\,
\widetilde{c}(d) \,\,\max_{1 \leq  k \leq m}\big\|\mathbf{Q}_k\colon \mathcal{B}^N_{=k} \to \mathcal{B}^N_{\mathcal{S}_{=k}}\big\|
\max_{1 \leq  k \leq d}  \boldsymbol{\lambda}\big(\mathcal{B}^N_{(\mathcal{S}_{=k})^\flat}\big)\,,
\]
where $\widetilde{c}(d) \leq   (d+1) (1+\sqrt{2})^d (2d) \kappa^d 2^{d}$ and $\mathbf{Q}_k$ denotes the projection annihilating Fourier
coefficients with indices $S$ not in~$\mathcal{S}_{=k}=\{S \in \mathcal{S} : \,|S|=k \}$.
\end{theorem}

Again we need preparation for the proof, and start with two elementary observations. We start with a result  taken from \cite[Lemma 22.2]{defant2019libro}.

\begin{lemma}
For every finite dimensional Banach  lattice  $X$, and every finite dimensional Banach space $Y$ one has
  \begin{equation*}\label{supo}
\boldsymbol{g\!l}\big( \mathcal{L}(X,Y)\big) \leq \boldsymbol{\lambda}(Y) \,.
    \end{equation*}
\end{lemma}

The second tool is an elementary piece of multilinear algebra.

\begin{lemma}\label{sym}
For $f \in\mathcal{B}_{=d}^N$ let $P_f$ be the  associated
$d$-homogeneous polynomial on $\mathbb{R}^N$ given by
\[
P_f(x)  = \sum_{1\leq j_1 < \dots < j_d \leq N} \hat{f}(\{j_1, \dots, j_d\}) x_{j_1} \ldots x_{j_d}\,.
\]
Then the unique $d$-linear symmetrization $\check{P}_f \colon (\mathbb{R}^N)^d \to \mathbb{R}$ of $P_f$ is given by
\[
\check{P}_f(u^1, \dots, u^d) =  \sum_{\substack{\bi \in  [N]^d \\   \text{$i_k\neq i_\ell$ for $k \neq \ell$}}}
\frac{\hat{f}(\{i_1, \dots, i_d\})}{d!} u^1_{i_1}\dots u^d_{i_d}\,.
\]
\end{lemma}

\begin{proof}
Let $(a_\bi)_{\bi \in \mathcal{M}(d,N)}$ be the symmetric matrix defining $\check{P}_f$. Then for all $x \in \mathbb{R}^N$
we have
\begin{align*}
\sum_{1 \leq j_1 < \dots < j_d \leq N} \hat{f}(\{j_1, \dots, j_d\}) x_\bj
=
\check{P}_f(x, \ldots, x)
=
\sum_{\bi \in \mathcal{M}(d,N)} a_\bi x_\bi
=
\sum_{\bj \in \mathcal{J}(d,N)} \sum_{\bi \in [\bj]}a_\bj x_\bj
=
\sum_{\bj \in \mathcal{J}(d,N)} |[\bj]| a_\bj x_\bj\,,
  \end{align*}
and  hence $\hat{f}(\{j_1, \dots, j_d\}) = d!  a_\bj$ whenever  $ 1 \leq j_1 < \dots < j_d \leq N$, and $= 0$ else.
    \end{proof}

\begin{proof}[Proof of Theorem~\ref{gl_versus_proj}]
To see the first statement, we  consider the following commutative diagram:
\begin{equation}\label{diagram}
  \xymatrix
{
 \mathcal{B}^N_{\mathcal{S}} \ar[rr]^{id} \ar[d]^{\text{$U_d$}} &  & \mathcal{B}^N_{\mathcal{S}}
 \\
\mathcal{L}\big(\ell^N_\infty,\mathcal{B}^N_{\mathcal{S}^\flat}\big)
\ar[r]_{I_d\;\;\;\;\;\;\;} & \mathcal{L}\big(\ell^N_\infty,\mathcal{P}_{d-1}( \ell^N_\infty)\big) \ar[r]_{\;\;\;\;\;\;\;\;\;\;\;\;\;\;V_d} &  \mathcal{P}_{d}(\ell^N_\infty), \ar[u]^{\mathbf{R}}
 \\
}
\end{equation}
where $\mathbf{R}$ is the canonical projection annihilating coefficients
with multi indices not generated by sets in $\mathcal{S}$, $I_d$ is the canonical isometric embedding and
\begin{align*}
&
\big(U_d(f)x\big)(u) := \check{P}_f( u, \ldots, u
,x)
\,\,\, \text{ for  } \,\,\, f \in  \mathcal{B}^N_{\mathcal{S}} \,\,\, \text{ and  } \,\,\, x \in \ell^N_\infty,u \in \{-1,+1\}^N \,,
\\[1ex]&
V_d(T)(y) := (Ty)(y) \,\,\, \text{ for  } \,\,\,  T \in \mathcal{L}\big( \ell^N_\infty,\mathcal{B}^N_{\mathcal{S}^\flat}\big)
\,\,\, \text{ and  } \,\,\, y \in  \ell^N_\infty.
\end{align*}
We show that $U_d$, as an operator from $\mathcal{B}^N_{\mathcal{S}}$ into $\mathcal{L}\big(\ell^N_\infty,\mathcal{B}^N_{\mathcal{S}^\flat}\big)$,
is well-defined. Indeed, by Lemma~\ref{sym} for $x \in \ell^N_\infty$ and $u \in \{-1,+1\}^N$
 \begin{align*}
 \check{P}_f( u, \ldots, u,x)
 &
 =
  \sum_{\substack{\bi \in  [N]^d \\   \text{$i_k\neq i_\ell$ for $k \neq \ell$}}}
\frac{\hat{f}(\{i_1, \dots, i_d\})}{d!} u_{i_1}\dots u_{i_{d-1}} x_{i_d}
\\&
  =
      \sum_{\substack{\bi \in  [N]^{d-1} \\   \text{$i_k\neq i_\ell$ for $k \neq \ell$}}}
      \bigg(\sum_{\substack{1 \leq  \ell \leq N\\ i_k \neq \ell}}
\frac{\hat{f}(\{i_1, \dots, i_d\})}{d!} x_{\ell}\bigg)   u_\bi
\\&
  =
  \sum_{1 \leq j_1 < \dots < j_d \leq N}
  (d-1)!    \bigg(\sum_{\substack{1 \leq  \ell \leq N\\ j_k \neq \ell}}
\frac{\hat{f}(\{j_1, \dots, j_d\})}{d!} x_{\ell}\bigg)   u_\bj
\\&
  =
  \sum_{S \in \mathcal{S}^\flat}
  (d-1)!  \,  \bigg(\sum_{\substack{1 \leq  \ell \leq N\\ j_k \neq \ell}}
\frac{\hat{f}(S)}{d!} x_{\ell}\bigg) \,\,  u^S\,.
 \end{align*}
By the  polarization formula from \cite[Proposition~4]{defant2019fourier} we have $\|U_d\| \leq 2d$, and moreover trivially  $\|V_d\| \leq~1$. Hence by the ideal property~\eqref{niceprop} of the Gordon--Lewis constant
\[
\boldsymbol{g\!l}\big(\mathcal{B}^N_{\mathcal{S}}\big) \leq \,\,2d \,\, \big\|\mathbf{R}\colon \mathcal{P}_{d}(\ell^N_\infty) \to \mathcal{B}^N_{\mathcal{S}}\big\|\, \,\boldsymbol{g\!l}\big(\mathcal{L}\big(\ell_\infty^N,\mathcal{B}^N_{\mathcal{S}^\flat}\big)\big)
\leq  \,\, 2d \,\, \|\mathbf{R}\colon \mathcal{P}_{d}(\ell^N_\infty) \to \mathcal{B}^N_{\mathcal{S}}\big\|\, \, \boldsymbol{\lambda}\big(\mathcal{B}^N_{\mathcal{S}^\flat}\big)\,,
\]
where for the last estimate we use  Lemma~\ref{supo}. Finally, since by Lemma~\ref{OrOuSe}
\begin{equation*}
\big\|\mathbf{R}\colon \mathcal{P}_{d}(\ell^N_\infty) \to \mathcal{B}^N_{\mathcal{S}}\big\|  \leq  \kappa^d 2^{d-1} \big\|\mathbf{Q}\colon \mathcal{B}^N_{=d} \to \mathcal{B}^N_{\mathcal{S}}\big\|\,,
\end{equation*}
the proof of  the first claim is complete.

For the second claim we assume that $\mathcal{S} \subset \big\{ S \subset [N] \colon |S| \leq d \big\}$, and consider the following commutative diagram
\begin{equation} \label{greatpic}
    \xymatrix
{
  \mathcal{B}^N_{\mathcal{S}}
  \ar[d]^{\text{$\mathbf{O}\oplus\bigoplus\mathbf{P}_{k}$}}
  \ar[r]^{\text{$ \id_{ \mathcal{B}^N_{\mathcal{S}}}$}}
 & \mathcal{B}^N_{\mathcal{S}}
 \\
  \mathbb{C} \oplus_\infty \bigoplus_\infty \mathcal{B}^N_{\mathcal{S}_{=k}}
  \ar[d]^{\text{$ \id_\mathbb{C}\oplus
  \bigoplus U_k$}}
  \ar[r]^{\text{$\id_\mathbb{C}\oplus\bigoplus \id_{ \mathcal{B}^N_{\mathcal{S}_{=k}}}$}}
 &
 \mathbb{C} \oplus_1\bigoplus_1 \mathcal{B}^N_{\mathcal{S}_{=k}} \ar[u]^{\sum }  &
 \mathbb{C} \oplus_1\bigoplus_1 \mathcal{P}_{k}(\ell^N_\infty) \ar[l]_{\text{$\id_\mathbb{C}\oplus\bigoplus\mathbf{R}_k$}}
 \\
 \mathbb{C} \oplus_\infty\bigoplus_\infty  \mathcal{L}\big(\ell^N_\infty,\mathcal{B}^N_{(\mathcal{S}_{=k})^\flat}\big)
\ar[r]^{\id_\mathbb{C}\oplus\bigoplus I_k}
&
 \,\,\,\mathbb{C} \oplus_\infty\bigoplus_\infty \mathcal{L}\big(\ell^N_\infty,\mathcal{P}_{k-1}(\ell^N_\infty)\big) \ar[r]^{\Phi}
 & \mathbb{C} \oplus_1\bigoplus_1 \mathcal{L}\big(\ell^N_\infty,\mathcal{P}_{k-1}(\ell^N_\infty)\big)
 \ar[u]^{\text{$\id_\mathbb{C}\oplus\bigoplus V_k$}}\,.
}
\end{equation}
Let us explain our notation in this diagram:
Note first that $\mathbf{P}_k$ and $\mathbf{R}_k$ stand for the canonical projections annihilating coefficients.
By Klimek's result already used in~\eqref{klimek}, for each $k$, we have
\[
\| \mathbf{P}_{k}\colon \mathcal{B}^N_{\mathcal{S}} \to \mathcal{B}^N_{\mathcal{S}=k} \| \leq (1+\sqrt{2})^d\,,
\]
and  from  Lemma~\ref{OrOuSe}, we conclude that
\begin{equation*}
 \big\|\mathbf{R}_k \colon \mathcal{P}_{k}(\ell^N_\infty) \to \mathcal{B}^N_{\mathcal{S}_{=k}}\big\|  \leq  \kappa^d 2^{d-1} \big\|\mathbf{Q}_k \colon \mathcal{B}^N_{=k} \to \mathcal{B}^N_{\mathcal{S}_{=k}}\big\|\,.
\end{equation*}
Note that $U_k$, $V_k$ and $I_k$ for each $1 \leq k \leq d$ are the operators from \eqref{diagram}. If $f = a_0 +\sum_{k=1}^d f_k$ is a~decomposition of $f \in \mathcal{B}^N_{\mathcal{S}}$, then
$\mathbf{O}(f) = a_0$, and hence  $\big(\mathbf{O}\oplus\bigoplus\mathbf{P}_{k}\big)(f) = \big(a_0, (f_k)_{k=1}^d\big)$. Additionally, $\sum$ is the mapping which assigns to every $\big(a_0, (f_k)_{k=1}^m\big)$ the polynomial $f = a_0 +\sum_{k=1}^d f_k$,  and $\Phi$ stands for the identity map. The notation for the rest of the maps is self-explaining.



Applying the ideal property~\eqref{niceprop}, we all  together arrive at
\[
\boldsymbol{g\!l}\big(\mathcal{B}^N_{\mathcal{S}}\big)
 \le 2d (d+1)(1+\sqrt{2})^d \,\max_{1 \leq  k \leq d}\big\|\mathbf{R}_k \colon \mathcal{P}_k(\ell_\infty^N) \to \mathcal{B}^N_{\mathcal{S}=k}\big\|\,\,
\boldsymbol{g\!l}\big(\mathbb{C} \oplus_{\infty} \bigoplus_\infty   \mathcal{L}\big(\ell^N_\infty,\mathcal{B}^N_{(\mathcal{S}_{=k})^\flat}\big)\big)\,.
\]
It is easy to check that for any Banach spaces $X$ and $Y$ one has $\boldsymbol{g\!l}\big(X \oplus_{\infty} Y\big) \leq
2\,\max\{\boldsymbol{g\!l}(X),\,\boldsymbol{g\!l}(Y)\}$. Thus, to complete the proof, it suffices to show that
\begin{equation*}\label{claimA}
\boldsymbol{g\!l}\big( \bigoplus_\infty   \mathcal{L}\big(\ell^N_\infty,\mathcal{B}^N_{(\mathcal{S}_{=k})^\flat}\big)\big)
 \leq
 \max_{1 \leq  k \leq m} \boldsymbol{\lambda}\big( \mathcal{B}^N_{(\mathcal{S}_{=k})^\flat}\big)\,.
\end{equation*}
Indeed, using standard properties of $\varepsilon$- and $\pi$-tensor products (see e.g., \cite{defant1992tensor}), we have
\begin{align*}
\bigoplus_\infty  \mathcal{L}\big(\ell^N_\infty,\mathcal{B}^N_{(\mathcal{S}_{=k})^\flat}\big) &
\hookrightarrow \bigoplus_\infty  \mathcal{L}\big(\ell^N_\infty,\bigoplus_\infty \mathcal{B}^N_{(\mathcal{S}_{=k})^\flat}\big) \\
& = \ell_\infty^d \otimes_\varepsilon \big[ \ell^N_1 \otimes_\varepsilon \bigoplus_\infty \mathcal{B}^N_{(\mathcal{S}_{=k})^\flat}\big] \\
& = \big[\ell_\infty^d \otimes_\varepsilon \ell^N_1 \big] \otimes_\varepsilon \bigoplus_\infty \mathcal{B}^N_{(\mathcal{S}_{=k})^\flat} \\
& = \big(\ell_1^d \otimes_\pi \ell^N_\infty\big)^\ast   \otimes_\varepsilon \bigoplus_\infty \mathcal{B}^N_{(\mathcal{S}_{=k})^\flat} =
\mathcal{L}\big(  \ell_1^d(\ell^N_\infty), \bigoplus_\infty \mathcal{B}^N_{(\mathcal{S}_{=k})^\flat}\big)\,,
\end{align*}
where the first space in fact is $1$-complemented in the second one, and all other identifications are isometries. Then we deduce from Lemma~\ref{supo} that
\[
\boldsymbol{g\!l}\big( \bigoplus_\infty  \mathcal{L}\big(\ell^N_\infty,\mathcal{B}^N_{(\mathcal{S}_{=k})^\flat}\big)\big)
 \leq
\boldsymbol{g\!l}\big(\mathcal{L}\big(  \ell_1^d(\ell^N_\infty), \bigoplus_\infty \mathcal{B}^N_{(\mathcal{S}_{=k})^\flat}\big)\big)
 \leq
 \boldsymbol{\lambda} \big(\bigoplus_\infty \mathcal{B}^N_{(\mathcal{S}_{=k})^\flat} \big)\,.
\]
Since
\[
\boldsymbol{\lambda} \big(\bigoplus_\infty \mathcal{B}^N_{(\mathcal{S}_{=k})^\flat} \big)
= \gamma_\infty \big(\id_{\bigoplus_\infty \mathcal{B}^N_{(\mathcal{S}_{=k})^\flat}} \big)
\leq \max_{1 \leq  k \leq m} \gamma_\infty \big(\id_{\mathcal{B}^N_{(\mathcal{S}_{=k})^\flat}} \big)
= \max_{1 \leq  k \leq m} \boldsymbol{\lambda} \big(\mathcal{B}^N_{(\mathcal{S}_{=k})^\flat} \big)\,,
\]
the proof is complete.
\end{proof}

\subsection{Sidon constants and the Bohnenblust-Hille inequality}
In Theorem~\ref{BGL3} we prove that the projection constant of $\mathcal{B}^N_{\mathcal{S} }$ and its Sidon constant
(see again \eqref{sidon}) are closely related. In the following result we  describe the asymptotic behaviour of $\sid (\mathcal{B}^N_{\mathcal{S} })$ in the case that $\mathcal{S}$ is 'big'.

\begin{proposition} \label{sid}
Given integers $1 \leq d \leq N$, let $\mathcal{S} \subset [N]$ be such that $(N/d)^{d/2}~\leq~|\mathcal{S}|$ and $|S| \leq d$
for all $S \in \mathcal{S}$. Then there are constants
$C_1$, $C_2 \ge 1$ (independent of $N,d,\mathcal{S}$) such that
\[
C_1 \, \, \frac{1}{\sqrt{N}}|\mathcal{S}|^{\frac{1}{2}}
\, \, \leq \,\, \boldsymbol{\sid} (\mathcal{B}^N_{\mathcal{S} }) \,\,\leq\,\,
C_2^{\sqrt{d \log d}} \frac{1}{\sqrt{N}}|\mathcal{S}|^{\frac{1}{2}}\,.
\]
\end{proposition}

For the proof of the upper bound we need  the so-called  subexponential Bohnen\-blust-Hille
inequality for functions on Boolean cubes from \cite[Theorem~1]{defant2018bohr}: There is a~constant $C \ge 1$ such that for each $1 \leq d \leq N$
and every $f \in \mathcal{B}_{\leq d}^{N}$ one has
\begin{equation}\label{equa:BHBoolean}
\Big(\sum_{\substack{S \subset [N]\\ |S|\leq d}}{|\widehat{f}(S)|^{\frac{2d}{d+1}}}\Big)^{\frac{d+1}{2d}}
\leq C^{\sqrt{d \log d}} \, \| f\|_{\infty}.
\end{equation}

\begin{proof}[Proof of Proposition~$\ref{sid}$]
The upper bound  follows from H\"older's inequality and \eqref{equa:BHBoolean}. That is,  for all functions
$f \in \mathcal{B}^N_{\mathcal{S}}$,
\[
\sum_{|S|\leq d} |\hat{f}(S)| \leq \bigg(  \sum_{|S|\leq d} |\hat{f}(S)|^{\frac{2d}{d+1}} \bigg)^{\frac{2d}{d+1}}
|\mathcal{S}|^{\frac{d-1}{2d}} \leq C^{\sqrt{d\log d }} \frac{1}{|\mathcal{S}|^{\frac{1}{d}}} |\mathcal{S}|^{\frac{1}{2}}\|f\|_\infty
\,.
\]
But by assumption
$
\sqrt{N} \leq \sqrt{d} |\mathcal{S}|^{\frac{1}{d}}\,,
$
and hence the claim follows from the definition of Sidon constants given in~\eqref{sidon}. The proof of the lower estimate is  probabilistic. Indeed, by the Kahane-Salem-Zygmund
inequality for the Boolean cube (see, e.g., \cite[Lemma~3.1]{defant2018bohr}) there is a family $(\varepsilon_S)_{S \in \mathcal{S}}$
of signs such that for $f =  \sum_{S \in \mathcal{S}}  \varepsilon_S \chi_S$ we have
\[
\|f\|_\infty \leq 6 \sqrt{\log 2} \,\,  \sqrt{N} \Big( \sum_{S \in \mathcal{S}}  |\varepsilon_S|^2  \Big)^{\frac{1}{2}}\,,
\]
and hence
\[
|\mathcal{S}| = \sum_{S \in \mathcal{S}}|\widehat{f}(S)|\leq  \boldsymbol{\sid}(\mathcal{B}^N_{\mathcal{S} }) \,\,  6 \sqrt{\log 2} \,\,\sqrt{N} |\mathcal{S}|^{\frac{1}{2}}\,.
\]
This completes the argument.
\end{proof}

\begin{corollary}
There are constants $C_1, C_2 > 0$ such that for each integer $1 \leq d \leq N$ one has
\[
C_1 \, \, \frac{1}{\sqrt{N}} \binom{N}{d}^{\frac{1}{2}}
\, \, \leq \,\, \boldsymbol{\sid} \big(\mathcal{B}^N_{=d}\big) \,\,\leq\,\,
C_2^{\sqrt{d \log d}} \frac{1}{\sqrt{N}} \binom{N}{d}^{\frac{1}{2}}
\]
and
\[
C_1 \, \, \frac{1}{\sqrt{N}} \left( \sum_{k=0}^d \binom{N}{k}\right)^{\frac{1}{2}}
\, \, \leq \,\, \boldsymbol{\sid}\big(\mathcal{B}^N_{\leq d}\big) \,\,\leq\,\,
C_2^{\sqrt{d \log d}} \frac{1}{\sqrt{N}} \left( \sum_{k=0}^d \binom{N}{k}\right)^{\frac{1}{2}}\,.
\]
\end{corollary}

\begin{proof}
Since
$
\Big(\frac{N}{d}\Big)^d \leq \binom{N}{d} \leq  \sum_{k=0}^d \binom{N}{k}
$
(see again  \eqref{ukraineA}), both  sets $\mathcal{S} = \{ S  \colon |S| =d \}$  and $\mathcal{S} = \{ S  \colon |S| \leq d \}$
satisfy the assumptions of Proposition~\ref{sid}.
\end{proof}

\begin{corollary}
There are constants $C_1, C_2 >0$ such that for all $1 \leq d \leq N$,
\[
C_1 \frac{1}{\sqrt{d}}\,\bigg( \frac{N}{d}  \bigg)^{\frac{d-1}{2}} \leq
\boldsymbol{\sid} \big(\mathcal{B}^N_{= d}\big) \leq \boldsymbol{\sid} \big(\mathcal{B}^N_{\leq d}\big) \leq
C_2^{\sqrt{d \log d}} e^{\frac{d}{2}} \,\frac{1}{\sqrt{d}}\,\bigg( \frac{N}{d}  \bigg)^{\frac{d-1}{2}}\,.
\]
In particular, we have the following hypercontractive comparison:
\[
\boldsymbol{\sid} \big(\mathcal{B}^N_{= d}\big) \sim_{C^d }
\bigg( \frac{N}{d}  \bigg)^{\frac{d-1}{2}}
\quad \text{ and } \quad
 \boldsymbol{\sid} \big(\mathcal{B}^N_{\leq d}\big) \sim_{C^d }
\bigg( \frac{N}{d}  \bigg)^{\frac{d-1}{2}}
\,.
\]
\end{corollary}

\begin{proof}
Both first and the third  estimate follow from the  preceding corollary. For the lower one  use again \eqref{ukraineA}, and for the upper
note that it suffices to check that
\[
\frac{1}{\sqrt{N}} \bigg( \sum_{k=0}^d \binom{N}{k}\bigg)^{\frac{1}{2}}
\leq e^\frac{d}{2} \frac{1}{\sqrt{d}}\,\bigg( \frac{N}{d}  \bigg)^{\frac{d-1}{2}}\,;
\]
indeed, this is another consequence of  \eqref{ukraineAA}.
\end{proof}


\end{document}